\documentclass[11pt]{amsart}
\usepackage{mathtools}
\usepackage{amsfonts}
\usepackage{mathrsfs}
\usepackage{bigints}
\usepackage{enumitem}
\usepackage[latin1]{inputenc}
\usepackage{amsmath,amssymb}
\usepackage{latexsym}
\usepackage{epstopdf}
\usepackage{geometry}   
\usepackage{xcolor}
\usepackage[active]{srcltx}
\usepackage{graphicx}
\usepackage{epstopdf}
\usepackage[
bookmarks=true,         
bookmarksnumbered=true, 
colorlinks=true, pdfstartview=FitV, linkcolor=blue, citecolor=blue,
urlcolor=blue]{hyperref}

\newtheorem {theorem}{Theorem}[section]
\newtheorem {proposition}{Proposition}[section]
\newtheorem {corollary}{Corollary}[section]

\newtheorem{lemma}{Lemma}[section]

\numberwithin{equation}{section}

\def\R{{\mathbb R}}

\def\wC{\widetilde{C}}
\def\wA{\widetilde{A}}

\def\wY{\widetilde{Y}}

\begin{document}

\title[Euler approximation for stochastic Volterra equations]{Error distribution of  the Euler approximation scheme for stochastic Volterra equations}

\author{David Nualart, Bhargobjyoti Saikia}
 \thanks{%
	The work by D. Nualart has been supported   by the  NSF grant DMS-2054735}

\address{David Nualart: Department of Mathematics, University of Kansas, 1450 Jayhawk Blvd., Lawrence, KS 66045. USA.}
\email{nualart@ku.edu}   

\address{Bhargobjyoti Saikia: Department of Mathematics, University of Kansas, 1450 Jayhawk Blvd., Lawrence, KS 66045. USA.}
\email{bhargob@ku.edu}

\date{\today}

\maketitle
 \begin{abstract} 
 The purpose of this paper is to establish the convergence in distribution of the normalized error in the Euler approximation scheme for stochastic Volterra equations driven by a standard Brownian motion,  with a kernel of the form $(t-s)^\alpha$, where $\alpha \in (-\frac 12, \frac 12)$.
 \end{abstract}
 
  \section{Introduction}\label{section settings}
  There have been several studies on stochastic Volterra equations starting from the pioneering work by Berger and Mizel \cite{BM1,BM2}. We could mention, for instance,  the papers by Protter \cite{Pro}, Pardoux and Protter \cite{PP}, Cochran, Lee and Potthoff \cite{CLP}, Coutin and Decreusefond \cite{CD}, Decreusefond \cite{D}, Wang \cite{Wa}, Zhang \cite{Zh}, and the references therein.

In this paper we are interested in the  following stochastic Volterra  equation
\begin{align} \label{volterra}
X_t=X_0+\int^t_0 (t-s)^{\alpha} \sigma(X_s)\, dW_s,  \qquad t\in[0,T],
\end{align}
where $X_0 \in \mathbb{R}$,  $\alpha\in(-\frac{1}{2}, \frac{1}{2})$,    $\sigma$   is a Lipschitz function and  
$W=\{W_t, t\ge [0,T]\}$ is a standard Brownian motion defined in a probability space $(\Omega, \mathcal{F}, P)$.
Our aim is to study the fluctuations of the  error in the Euler approximation scheme for this equation. It is known that the rate of convergence of the Euler approximation scheme is of the order $n^{-\alpha-\frac 12}$. We refer for this to the work by  Li, Huang and Hu \cite{LHH}, where they consider the more general  $\theta$-Euler  and   Milstein schemes    in the singular case $-\frac 12<\alpha <0$, and the papers by
Zhang \cite{Zh1} and Richard, Tan and Yang \cite{RTY}, where  the Euler approximation scheme for more general stochastic Volterra equations is studied. 
However, to the best of our knowledge, there have been no  results on the convergence in distribution of the normalized error, analogous to the case of ordinary stochastic differential equations (see, for instance, the work by Jacod and Protter \cite{JP}).

The Euler scheme approximation for the solution to equation (\ref{volterra}) is defined as
\begin{align}\label{euler}
X^n_t=X_0+\int^t_0(t-\eta_n(s))^{\alpha} \sigma(X^n_{\eta_{n}(s)})\, dW_s,
\end{align}
where $\eta_n(s)=\frac{\lfloor ns\rfloor}{n}$.  That is, $\eta_n(s) = \frac jn$ if $\frac jn \le s< \frac {j+1}n$ for some integer $j\ge 0$. 
Consider the   normalized error defined by
 \begin{equation} \label{merror}
 Y^n_t= n^{\alpha+\frac 12} ( X^n_t-X_t).
 \end{equation}
It is known that    $Y^n_t$ is bounded in $L^p$, for any $p\ge 2$, uniformly in $t\in [0,T]$ and $n \ge1$ (see  Lemma \ref{bound}) for a short proof. 
This means that the rate of convergence to zero  as $n \to \infty$ in $L^p$  of $X^n_t-X_t$ is of the order $n^{-\alpha-\frac 12}$.
 The goal of this paper is to study the asymptotic behavior of the normalized error \eqref{merror}.
 
We will describe here  the approach and the challenging difficulties encountered in handling the limit distribution of the normalized error
$Y^n_t$.  This error can be decomposed as follows:
\begin{align} \nonumber
Y^n_t &= n^{\alpha+\frac 12}\int^t_0   [(t-\eta_n(s))^{\alpha} - (t-s)^{\alpha}]    \sigma(X^n_{\eta_{n}(s)})\, dW_s\\ \nonumber
& \qquad +n^{\alpha+\frac 12}\int^t_0    (t-s)^{\alpha}   [   \sigma(X^n_{\eta_n(s)}) -\sigma(X_s)]\, dW_s\\ \label{decom}
&:= A^n_t +Y^{n,1}_t.
\end{align}
The term $A^n_t$ has an unexpected asymptotic behavior. Indeed, it only converges at points of the form $t=\eta_n(t)$
(see Remark 1 below) and its limit in  law is
$\kappa_1 \sigma(X_t) Z_t$, where $\kappa_1$  is a constant defined in \eqref{kappa1} and $Z_t$ is a $N(0,1)$ random variable independent of the Brownian motion $W$. Moreover,  (see Proposition \ref{first_term}) the finite-dimensional distributions of the process $\{A^n_{\eta_n(t)}, t\in [0,T]\}$ converge to those of
$\{ \kappa_1 \sigma(X_t) Z_t, t\in [0,T]\}$, where $\{ Z_t, t\in [0,T]\}$ is a centered Gaussian process independent of $W$ with covariance
\begin{equation} \label{covZ}
E[ Z_{t_1} Z_{t_2}]=\mathbf{1}_{ \{t_1  =t_2\}}.
\end{equation}
Notice that the limit in distribution of  $A^n_{\eta_n(t)}$ does  not have any H\"older continuity properties, which creates problems in the  global analysis of the error. For this reason, we have split the study of the error in two parts, corresponding to both terms in the above decomposition.
For the term  $Y^{n,1}_t$, the limit in distribution can be expressed as the solution to a stochastic Volterra equation, with an input given by a stochastic integral term with respect to an independent Brownian motion $B$. A precise statement is given in Theorem  \ref{mainthm}.
The proof of the convergence of the term $Y^{n,1}_t$ relies on the application of the  asymptotic version of Knight's theorem. Finally, in Theorem \ref{thmjoint} we establish the convergence in distribution of the couple of random variables   $(A^n_{\eta_n(t)}, Y^{n,1}_t)$, which implies, as a consequence, the limit in law of the normalized error  $Y^n_{\eta_n(t)}$  given in Corollary \ref{cor}.

The paper is organized as follows. Section 2 contains the main result. In Section 3 we introduce some preliminaries and in Section 4 we present the proof of the main result. Some technical auxiliary results are included in the Apppendix.

\section{Main result}
  Let $X=\{ X_t , t\in [0,T]\}$ be the solution to the stochastic Volterra equation \eqref{volterra}.
 Consider the Euler approximation scheme   $X^n=\{ X^n_t , t\in [0,T]\}$ defined in \eqref{euler}.
The following theorem, which deals with the convergence of the processes $A^n_t$ and  $Y^{n,1} _t$ appearing in the decomposition of the normalized error,  is the main result of this paper.

\begin{theorem}\label{mainthm}
Let $Y^n=\{ Y^n_t, t\in [0,T]\}$ be the normalized error  in the Euler approximation scheme of equation \eqref{volterra} defined in  \eqref{merror}.
Suppose that $\sigma$ is continuously differentiable and $\sigma'$ is H\"older continuous of order $\beta$ for some $\beta \in (0,1]$.
Consider the decomposition introduced in \eqref{decom}. Then  the following results hold true:
\begin{itemize}
\item[(i)]   As $n$ tends to infinity,  the finite-dimensional distributions of  the process 
$\{A^n_{\eta_n(t)}, t\in [0,T]\}$ converge in distribution  to those of  $\kappa_1 \sigma(X_t) Z_t$, where
\begin{equation} \label{kappa1}
\kappa_1=  \sqrt{\sum_{k=1}^\infty \int_0^1 [ k^\alpha - (k-x)^\alpha]^2 dx}
\end{equation}
 and $\{Z_t, t\in [0,T]\}$  is a zero mean Gaussian process independent of $W$ with covariance  \eqref{covZ}.
 
\item[(ii)]  As $n$ tends to infinity,   the finite-dimensional distributions of  the process  $\{Y^{n,1} _t, t\in [0,T]\}$  converge in distribution to those of $\{Y^{\infty,1} _t, t\in [0,T]\}$, where  $Y^{\infty,1} _t$ is the solution to the following stochastic  Volterra equation
\[
Y^{\infty,1} _t = \kappa_2 \int_0^t (t-s)^\alpha (\sigma'\sigma)(X_s) dB_s + \int_0^t (t-s)^\alpha \sigma'(X_s) Y^{\infty,1}_s dW_s,
\]
where $B=\{B_t, t\in [0,T]\}$ is a standard Brownian motion independent of $W$ and 
\begin{equation} \label{kappa2}
\kappa_2 =\sqrt{ \kappa_3 + \frac {2\alpha+1} {2(\alpha+1)^2 } +\kappa_4 +\kappa_5},
\end{equation}
 with
\begin{equation} \label{kappa3}
\kappa_3=  \sum_{k=1}^ \infty \int_{0}^1
\left[(x+k)^{\alpha}-k^{\alpha}\right]  ^2dx,
\end{equation}
\begin{equation} \label{kappa4}
\kappa_4=
\sum_{k= 0}^{\infty}\int_0^1   \int_{0} ^{1 \wedge (y+k)} [ (y+k)^\alpha - (y+k-z)^\alpha]^2 dzdy.
\end{equation}
and
\begin{equation} \label{kappa5}
\kappa_5= \sum_{k=1}^{\infty}  \int_0^1  \int_{0} ^{1} [(y+k)^\alpha- (y+k-z)^\alpha][ (y+k)^\alpha -k^\alpha] dz dy.
\end{equation}
\end{itemize}
\end{theorem}
 
 In the next result we analyze the joint convergence of $A^n_t$ and  $Y^{n,1} _t$ for a fixed $t\in [0,T]$.
 \begin{theorem} \label{thmjoint}
 Under the assumptions of Theorem \ref{mainthm}, for any $ t\in [0,T]$, the two-dimensional random variable
 $(A^n_{\eta_n(t)}, Y^{n,1}_t)$ converges in distribution to $(\kappa_1 \sigma(X_t) Z_t, Y^{\infty,1}_t)$, where $Z_t$ and $Y^{\infty,1}_t$ are the random variables described in points (i) and (ii) of Theorem  \ref{mainthm} and $Z_t$ and $Y^{\infty,1}_t$  are independent. 
 \end{theorem}
 
 As a consequence of theorems \ref{mainthm} and \ref{thmjoint} we obtain the following result.
 \begin{corollary} \label{cor}
  Under the assumptions of Theorem \ref{mainthm}, for any $ t\in [0,T]$,  the normalized  error  converges in distribution to
 $ \kappa_1 \sigma(X_t) Z_t+ Y^{\infty,1}_t$, where $Z_t$ and $Y^{\infty,1}_t$ are the random variables described in points (i) and (ii) of Theorem  \ref{mainthm} and $Z_t$ and $Y^{\infty,1}_t$  are independent. 

 \end{corollary}

\section{Preliminaries}
In this section we will present  some basic properties of the solution to the stochastic Volterra equation \eqref{volterra} and its associated Euler approximation scheme \eqref{euler}.
Throughout this paper   $C$ will denote  a generic constant, which may depend on  $\alpha$, the coefficient $\sigma$ and the terminal time $T$.
Also,  $C_p$ will denote  a generic constant, which may depend on  $p$,  $\alpha$, the coefficient $\sigma$ and the terminal time $T$.
Moreover, we will denote by $\{\mathcal{F}_t, t\in [0,T]\}$ the increasing family of $\sigma$-algebras generated by the Brownian motion $W$.
 
 Assuming that the coefficient $\sigma$ is Lipschitz, there exist a unique solution to equation \eqref{volterra}, such that
 for any $p\ge 2$,
 \begin{equation}  \label{sup}
 \sup_{0\leq t\leq T} E[|X_t|^p]<C_p.
 \end{equation}
 Furthermore,  for any $p\ge 2$ and for any $s,t\in [0,T]$,
 \begin{equation}  \label{rgla}
  E[ |X_t-X_s|^p] \leq C_p |t-s|^{p(\alpha+\frac 12)}.
\end{equation}
We refer, for instance, to Theorem 2.1 and Theorem 2.2 of  \cite{LHH} for the proof of these estimates in the case $\alpha<0$.  The case $\alpha\ge 0$ is easier and can be obtained by similar arguments.

The Euler approximation $X^n_t$ satisfies similar estimates. That means, for any $p\ge 2$, we have
 \begin{equation}  \label{sup2}
 \sup_{0\leq t\leq T} E[|X^n_t|^p]<C_p
 \end{equation}
and  for any $s,t\in [0,T]$,
 \begin{equation}  \label{rglo}
  E[ |X^n_t-X^n_s|^p] \leq C_p |t-s|^{p(\alpha+\frac 12)}.
\end{equation}
Again, these estimates can be obtained  as in the proof of  Theorem 2.1 and Theorem 2.2 of  \cite{LHH}, when $\alpha<0$, the case $\alpha\ge 0$ being simpler.
 
The following result provides the rate of convergence in  $L^p$, for any $p\ge 2$, for the Euler approximation scheme. The case $\alpha<0$ and $t=\eta_n(t)$ was established in Theorem 2.3 of \cite{LHH}. For the sake of completeness we include the main ideas of the proof.
\begin{lemma} \label{bound}
Let $Y^n_t$ be the normalized  error defined in \eqref{merror}. Then,  there exists a constant $C_p>0$  such that for any $t\in [0,T]$, and for any $p\ge 2$, we have
\begin{equation}  \label{rate}
E[| Y^n_t|^p] \le C_p.
\end{equation}
\end{lemma}

\begin{proof}
We can write, using Burkholder  and Minkowski's inequalities, the Lipschitz property of $\sigma$ and   the estimate \eqref{sup},
\begin{align*}
\| Y^n_t \|_p^2 &\le C_p n^{2\alpha+1}  \int_0^t   \| (t-s)^\alpha \sigma(X_s) - (t-\eta_n(s))^\alpha \sigma(X_{\eta_n(s)})  \|_p^2ds \\
& \le C_p   \int_0^t   (t-s)^\alpha  \| Y^n_s \|_p^2  ds+ C_p n^{2\alpha+1}
\int_0^t   [ (t-s)^\alpha   - (t-\eta_n(s))^\alpha  ]^2ds.
\end{align*}
Lemma \ref{lema1}  allows us to write
\[
\| Y^n_t \|_p^2  \le  C_p   \int_0^t   (t-s)^\alpha  \| Y^n_s \|_p^2  ds +C,
\]
which implies the desired estimate by an application of  an extended version of  Gronwall's inequality (see, for instance, \cite{Bru,YGD}).
\end{proof}

As a consequence of \eqref{rgla} and \eqref{rate}, we obtain, for any $s,t \in [0,T]$,
\begin{equation} \label{ineq7}
 E [|X^n_t-X_s|^p] \leq C_p \left( |t-s|^{p(\alpha+\frac 12)} +n^{-p(\alpha+\frac 12)}\right).
 \end{equation}

\section{Proof of  Theorem \ref{mainthm}}
Recall that $X= \{X_t, t \in [0,T]\}$ is the solution to the stochastic Volterra equation  \eqref{volterra} and the normalized
Euler scheme approximation for $X$  is defined in \eqref{merror}.

In order to prove Theorem \ref{mainthm} we make the following decomposition of the  normalized error
\[
 Y^n_t =  A^n_t  +Y^{n,1}_t=:A^n_t + C^n_t + Z^n_t,
\]
 where
  \begin{equation} \label{A}
  A^n_t= n^{\alpha+\frac 12}\int^t_0   [(t-\eta_n(s))^{\alpha} - (t-s)^{\alpha}]    \sigma(X^n_{\eta_{n}(s)})\, dW_s,
  \end{equation}
  \begin{equation} \label{C}
   C^n_t=n^{\alpha+\frac 12}\int^t_0    (t-s)^{\alpha}   [   \sigma(X^n_{\eta_n(s)}) -\sigma(X^n_s)]\, dW_s
   \end{equation}
   and
   \begin{equation} \label{Z}
   Z^n_t= n^{\alpha+\frac 12}\int^t_0 (t-s)^{\alpha}   [ \sigma(X^n_{s})- \sigma(X_{s})]\, dW_s.
 \end{equation}
 Then, the proof of Theorem \ref{mainthm} will be done in several steps included in the next subsections.
 
 \subsection{Modified error}
 The next proposition will allow us to replace the  three terms   $A^n_t$, $C^n_t$ and $Z^n_t$ 
 defined in \eqref{A}, \eqref{C} and \eqref{Z}, respectively, by simpler terms that have the same asymptotic behavior.
  \begin{proposition}\label{clean}
Let $A^n_t$, $C^n_t$ and $Z^n_t$ be defined in \eqref{A}, \eqref{C} and \eqref{Z}, respectively.  Then we claim that the following convergences are true 
\begin{equation} \label{A1}
\lim_{n\rightarrow \infty}  \sup_{t\in [0,T]}E[ |A^n_t- \widetilde{A}^n_t|^2] =0,
\end{equation}
\begin{equation} \label{C11}
\lim_{n\rightarrow \infty}  \sup_{t\in [0,T]}E[ |C^n_t- \widetilde{C}^n_t|^2] =0
\end{equation}
and
\begin{equation} \label{Z1}
\lim_{n\rightarrow \infty}  \sup_{t\in [0,T]} E[ |Z^n_t- \widetilde{Z}^n_t|^2] =0,
\end{equation}
where
\begin{equation} \label{A2}
  \widetilde{A}^n_t=n^{\alpha+\frac 12} \sigma(X_t)\int^t_0   [(t-\eta_n(s))^{\alpha} - (t-s)^{\alpha}]dW_s,
  \end{equation}
 \begin{equation} \label{C2}
\widetilde{C}^n_t=n^{\alpha+\frac 12}\int^t_0    (t-s)^{\alpha}     \sigma'(X_s)\left(X^n_{\eta_n(s)} -X^n_s\right)\, dW_s
\end{equation}
and
 \begin{equation} \label{Z2}
 \widetilde{Z}^n_t=n^{\alpha+\frac 12} \int^t_0 (t-s)^{\alpha}   \sigma'(X_s)\left(X^n_{s}-X_{s}\right)\, dW_s.
 \end{equation}
 \end{proposition}

\begin{proof}
The proof will be done in three steps.

\medskip \noindent
{\it Proof of  \eqref{A1}:}
Fix $\delta > 0$. With the notation
\begin{equation}
\label{psi1}\psi_{n,1}(s,t):=(t-\eta_n(s))^{\alpha} - (t-s)^{\alpha}, \quad 0\le s<t \le T,
\end{equation} 
we can write
\begin{align*}
 A^n_{t}-  \widetilde{A}^n_t
&=n^{\alpha+\frac 12} \left(\int_0^{t}\psi_{n,1}(s,t)\sigma(X^n_{\eta_{n}(s)})\, dW_s-\sigma(X_t)\int_0^{t}\psi_{n,1}(s,t)dW_s \right)\\
&= n^{\alpha+\frac 12}  \Bigg(\int_0^{(t-\delta)_+}\psi_{n,1}(s,t)  
\sigma(X^n_{\eta_{n}(s)}) dW_s -\sigma(X_t)\int_0^{(t-\delta)_+ }\psi_{n,1}(s,t)dW_s  \\
&\qquad+ \int_{(t-\delta)_+}^{t}\psi_{n,1}(s,t)   \left[
\sigma(X^n_{\eta_n(s)} -\sigma(X_{(t-\delta)_+}) \right] dW_s \\
& \qquad  + \left[\sigma(X_{(t-\delta)_+})-\sigma(X_t)\right]\int_{(t-\delta)_+}^{t}\psi_{n,1}(s,t)dW_s \Bigg).
\end{align*} 
    Now, using triangle inequality and Cauchy-Schwartz inequality, we obtain
    \begin{align*}
E[ |A^n_{t}-\widetilde{A}^n_t |^2]&\leq C n^{2\alpha+1} \int_0^{(t-\delta)_+}\psi^2_{n,1}(s,t)  
E(\sigma^2(X^n_{\eta_{n}(s)})  )ds \\
& \qquad +C  n^{2\alpha+1} \left[E\left(\sigma^4(X_t)\right)E\left[
\left(\int_0^{(t-\delta)_+}\psi_{n,1}(s,t)dW_s\right)^4\right]\right]^{\frac 12}\nonumber\\
&\qquad+C n^{2\alpha+1} \int_{(t-\delta)_+}^{t}\psi^2_{n,1}(s)   E\left[
\left| \sigma(X^n_{\eta_n(s)}) -\sigma(X_{(t-\delta)_+})  \right|^2\right] ds\\
&\qquad  +C n^{2\alpha+1} \left[ E\left ( \left| \sigma(X_{(t-\delta)_+})-\sigma(X_t)\right|^4 \right)E\left( \left| \int_{(t-\delta)_+}^{t}\psi_{n,1}(s,t)dW_s\right|^4\right) \right]^{\frac12}\\
&=: C n^{2\alpha+1} (A_1+A_2+A_3+A_4).
\end{align*}
By the estimate  \eqref{sup2}, taking into account that $\sigma$ is Lipschitz, we have 
\[
L_1:=\sup_{n\ge 1} \sup_{s\in [0,T]} E(\sigma^2(X^n_{\eta_{n}(s)})  ) <\infty.
\]
 Therefore, applying the estimate \eqref{ineq2}, we can write
\begin{equation} \label{e3}
A_1 \leq  L_1\int_0^{(t-\delta)_+} \psi_{n,1}^2(s,t)  ds \leq  L_1 C_2 n^{-2}\delta^{2\alpha-1}.
\end{equation}
 For the term $A_2$,  using that the random variable  $\int_0^{(t-\delta)_+}\psi_{n,1}(s,t)dW_s$ is Gaussian and applying again the estimate  \eqref{ineq2}, we obtain
\begin{align}
A_2 &\le \sqrt{3} \sup_{t\in [0,T]}  \left[E\left(\sigma^4(X_t)\right) \right]^\frac 12  \int_0^{(t-\delta)_+} \psi_{n,1}^2(s,t)  ds  \nonumber  \\
& \le  \sqrt{3} \sup_{t\in [0,T]}  \left[E\left(\sigma^4(X_t)\right) \right]^\frac 12 C_2  n^{-2}\delta^{2\alpha-1}. \label{e4}
\end{align}
  For  the term $A_3$, using the Lipschitz property of $\sigma$  and inequality \eqref{ineq7}, we obtain, assuming $\delta > \frac 1n$,
   \begin{align*}
   E\left[ \left|
\sigma(X^n_{\eta_n(s)}) -\sigma(X_{(t-\delta)_+}) \right|^2 \right]&\leq C \left[\sup_{s\in [(t-\delta)_+,t]}|\eta_n(s)-(t-\delta)_+| \right]^{2\alpha+1}+\frac{C}{n^{2\alpha+1}}\\
&\leq C \delta^{2\alpha+1}+\frac{C}{n^{2\alpha+1}}\\& \leq C \delta^{2\alpha+1}. \end{align*}
Therefore, using \eqref{ineq1},
 \begin{equation} \label{e5}
 A_3\leq C \delta^{2\alpha +1} \int_{0}^{t}\psi_{n,1}^2(s,t) ds
 \leq C \delta^{2\alpha+1} n^{-2\alpha-1}.
 \end{equation}
Finally, for the term $A_4$, using the Lipschitz property of $\sigma$ and the estimates \eqref{sup} and  \eqref{ineq1}, yields
 \begin{equation} \label{e6}
     A_4\leq C  \delta^{2\alpha+1}\int_{0}^{t}\psi_{n,1}^2(s,t)  ds \leq C  \delta^{2\alpha+1}n^{-2\alpha-1}.
 \end{equation}
The estimates \eqref{e3}, \eqref{e4}, \eqref{e5} and \eqref{e6} allow us to write
\[
\limsup_{n\rightarrow\infty} \sup_{t\in [0,T]}E\left[ |A^n_{t}- \widetilde{A}^n_t|^2\right]\leq C\delta^{2\alpha+1},
\]
where the constant $ C$ does not depend on $\delta$. Taking into account that $\delta>0$ is arbitrary, we conclude the proof of  \eqref{A1}.

\medskip \noindent
{\it Proof of  \eqref{C11}:}
We can write 
\[ 
C^n_t=n^{\alpha+\frac 12} \int^{t}_0(t-s)^{\alpha} \Psi^n_s\,   \left(X^n_{\eta_n(s)} -X^n_s\right)\, dW_s,
\]
where $ \Psi^n_s=\int^1_0 \sigma'(\theta X^n_{\eta_{n}(s)}+(1-\theta)X^n_s)\, d\theta$.
 Since $\sigma'$ is $\beta$-H\"older continuous, using the definition of $\Psi^n_{s}$ we have
\begin{equation}
\label{Eq1}|\Psi^n_s-\sigma'(X^n_{s})|\leq C|X^n_{\eta_n(s)}-X^n_{s}|^\beta.
\end{equation}
Now using \eqref{rglo} and the fact that $|\eta_n(s)-s|\leq \frac{1}{n}$ we have
\begin{equation}
\label{Eq2}E[|\Psi^n_{s}-\sigma'(X_s) |^p ]\leq C n^{-p\beta(\alpha+\frac12)}.
\end{equation}
 Next coming back to proof of  \eqref{C11}, we can write, using Cauchy-Schwartz inequality,
 \begin{align*} 
 E\left[ \left| C^n_t-\widetilde{C}^n_t\right|^2 \right]& = n^{2\alpha+1}E  \left[ \left|\int^t_0 (t-s)^{\alpha}     \left(\sigma'(X_s)-\Psi^n_s\right)\left(X^n_{\eta_n(s)} -X^n_s\right)\, dW_s\right|^2 \right]\\
 &= n^{2\alpha+1}\int^t_0 (t-s)^{2\alpha}    E \left[ \left(\sigma'(X_s)-\Psi^n_s\right)^2\left(X^n_{\eta_n(s)} -X^n_s\right)^2 \right]\, ds\\
 &\leq n^{2\alpha+1}  \int^t_0 (t-s)^{2\alpha}     \left(E[(\sigma'(X_s)-\Psi^n_s)^4]\right)^{\frac12}\left(E[(X^n_{\eta_n(s)} -X^n_s)^4]\right)^{\frac 12}ds.
 \end{align*}
Finally, from  the estimates \eqref{rglo} and  \eqref{Eq2}, we obtain
\[
E\left[ \left| C^n_t-\widetilde{C}^n_t\right|^2 \right] \le C n^{-\beta(2\alpha+1)},
\]
which completes the proof of \eqref{C11}.

\medskip \noindent
{\it Proof of  \eqref{Z1}:}
Rewrite $Z^n_t  $ as $Z^n_t=n^{\alpha+\frac 12} \int^t_0 (t-s)^{\alpha} \Psi^n_s  ( X^n_{s}-X_{s})\, dW_s$. Then, we can write, using the estimates \eqref{rate} and  \eqref{Eq2},
\begin{align*} 
E\left[Z^n_t- \widetilde{Z}^n_t \right] & =n^{2\alpha+1}E  \left[ \left| \int^t_0 (t-s)^{\alpha}     \left(\sigma'(X_s)-\Psi^n_s\right)\left(X^n_{s} -X_s\right)\, dW_s\right|^2 \right]\\
&\leq n^{2\alpha+1}  \int^t_0 (t-s)^{\alpha}     \left(E[(\sigma'(X_s)-\Psi^n_s)^4]\right)^{\frac12}\left(E[(X^n_{s} -X_s)^4]\right)^{\frac 12}\, ds\\
&\leq  C n^{-\beta(2\alpha+1)}.
\end{align*}
This completes the proof of \eqref{Z1}.
 \end{proof}

  Let us denote by $\widetilde{Y}^n_t$ the solution to the following stochastic integral equation
 \begin{align}\label{reformed_error}
 \widetilde{Y}^n_t =\widetilde{A}^n_t + \widetilde{C}^n_t +\int_0^t(t-s)^{\alpha}\sigma'(X_s)\widetilde{Y}^n_s\,dW_s.
 \end{align}
 The next proposition tell us that the limit behavior of the normalized error  $Y^n_t$ is analogous to the limit of 
 $ \widetilde{Y}^n_t$.
 
 \begin{proposition}  \label{propY}
 Let $\wY^n$ be defined as in \eqref{reformed_error}. Then, we have
  \[
  \lim_{n\rightarrow \infty} \sup_{t\in [0,T] }E[|\widetilde{Y}^n_t -Y^n_t|^2] =0.
  \]
 \end{proposition}
 
 \begin{proof}
 By Proposition  \ref{clean}, we can write
 \[
n^{\alpha+\frac 12}Y^n_t = R^n_t+  \widetilde{A}^n_t + \widetilde{C}^n_t + \widetilde{Z}^n_t,
\]
where $\lim_{n\rightarrow \infty} \sup_{t\in [0,T]} E[ |R^n_t|^2]=0$.  As a consequence,
\[
 \widetilde{Y}^n_t - Y^n_t =R^n_t +\int_0^t(t-s)^{\alpha}\sigma'(X_s)(\widetilde{Y}^n_s-Y^n_s)\,dW_s,
 \]
 which allow us to conclude easily the proof of the proposition.
\end{proof}

\subsection{Asymptotic behavior of $\wA^{n}_t$}
The term  $\wA^{n}_t$ is singular in the sense that for any fixed $t\ge 0$ we cannot guarantee that  $\wA^{n}_t$  converges in $L^2$. However, 
$\wA^{n}_{\eta_n(t)}$ does converges in $L^2$, as it is explained in the remark after the proof of the next proposition.   
The main result of this section is the following proposition.

\begin{proposition}\label{first_term}
The finite-dimensional distributions of the process   $\{\widetilde{A}^n_{\eta_n(t)} , t\in [0,T]\} $ converge in distribution, as $n\to \infty$,  to    those of
$ \{\kappa_1 Z_t \sigma(X_t), t\in [0,T]\}$, where $\kappa_1$ has been defined in \eqref{kappa1}  
 and $\{Z_t, t\in [0,T]\}$ is a  zero mean Gaussian process  independent of  $W$ with covariance given by \eqref{covZ}.
 \end{proposition}
 
\begin{proof}
Let $\psi_{n,1}(s,t)$ be defined as in \eqref{psi1} and set
\begin{equation} \label{ene}
N_t^{n,1}=  n^{\alpha+\frac 12} \int_0^t \psi_{n,1}(s,t) dW_s.
\end{equation}
It suffices to show that the finite-dimensional distributions of the process $\{ N^{n,1}_{\eta_n(t)}, t\in [0,T]\}$ converge in distribution to those of 
$\{Z_t, t\in [0,T]\}$. 
Notice that  $\{N^{n,1}_{\eta_n(t)}, t\in [0,T]\}$ is  a centered Gaussian  process. So, we just need to show that the covariance of $\{ N^{n,1}_{\eta_n(t)}, t\in [0,T]\}$ converges to the covariance of $\{Z_t, t\in [0,T]\}$ and also that the processes   $\{ N^{n,1}_{\eta_n(t)}, t\in [0,T]\}$ and $W$ are not correlated.

Let us first compute the variance of  $N^{n,1}_{\eta_n(t)}$. We have, with the change of variables $ns=x$ and $x-j=y$,
  \begin{align}
E[(N^{n,1}_{\eta_n(t)})^2]  \label{problem}
&= n^{2\alpha +1}
 \int_0^{\eta_n(t)} \psi_{n,1}^2(s, \eta_n(t))  ds\\  \nonumber
&=  n^{2\alpha +1} \sum^{\lfloor nt\rfloor-1}_{j=0}\int^{\frac {j+1} n}_{\frac jn } \left[ \left ( \frac {\lfloor nt\rfloor}n -\frac jn  \right)^{\alpha}-\left(\frac {\lfloor nt\rfloor}n-s \right)^{\alpha}\right]^2\, ds\\ \nonumber
&=  \sum^{\lfloor nt\rfloor-1}_{j=0}\int^{ j+1  }_{j } \left[(  \lfloor nt\rfloor -j )^{\alpha}-(\lfloor nt\rfloor -x)^{\alpha}\right]^2\, dx\\ \nonumber
&=  \sum^{\lfloor nt\rfloor-1}_{j=0}\int^{ 1  }_{0 } \left[(  \lfloor nt\rfloor -j )^{\alpha}-(\lfloor nt\rfloor -j-y)^{\alpha}\right]^2\, dy\\ \nonumber
&=\sum^{\lfloor nt\rfloor}_{k=1}\int^{1}_{0}  (k^\alpha - (k-y)^\alpha)^2\, dy,  \nonumber
\end{align}
which converges to $\kappa_1^2$ as $n\to \infty$, where $\kappa_1$ is defined in \eqref{kappa1}.  
 On the other hand, for any interval $[a,b]$, with $0\le a< b \le T$, we have
\[
\lim_{n\rightarrow \infty} E[N^n_{\eta_n(t)}(W_b-W_a)] =0.
\]
Indeed, with the same computations as above, we can write
\begin{align*}
|E[N^n_{\eta_n(t)}(W_b-W_a)]|  
&=n^{\alpha+\frac 12} \left|  \sum^{\lfloor nt\rfloor-1}_{j=0}\int_{[\frac jn, \frac {j+1}n]\cap [a,b]} \left[\left(\frac{\lfloor nt\rfloor}{n}-\frac{j}{n}\right)^{\alpha}-\left(\frac{\lfloor nt\rfloor}{n}-s\right)^{\alpha}\right]\, ds  \right|\\   
&=n^{-\frac 12}  \sum^{\lfloor nt\rfloor}_{k=1}\int^{1}_{0}  |k^\alpha - (k-y)^\alpha|\, dy\\
& \le C n^{ \max\{ \alpha, 0\} - \frac 12},  
\end{align*}
which converges to zero as $n\to \infty$.   Finally, for any $0\le t_1<t_2 \le T$ we have
\begin{align*}
|E[N^n_{\eta_n(t_2)} N^n_{\eta_n(t_1)} ]|  
&= n^{2\alpha+1} \left|    \int_0^{\eta_n(t_1)} \psi_{n,1}(s, \eta_n(t_1))  \psi_{n,1}(s, \eta_n(t_2))ds\right|\\
& \le
 n  \sum^{\lfloor nt_1\rfloor-1}_{j=0}\int_{\frac jn}^{ \frac {j+1}n }  |  (\lfloor nt_2\rfloor-j)^{\alpha}-(\lfloor nt_2\rfloor -ns)^{\alpha} | \\
 & \qquad \times |(\lfloor nt_1\rfloor -j)^{\alpha}-(\lfloor nt_1\rfloor-ns)^{\alpha} |
 ds.
 \end{align*}
The change of variables $ns=x$ and $x-j=y$ allow us to write
\begin{align*}
|E[N^n_{\eta_n(t_2)} N^n_{\eta_n(t_1)} ]|   &  \le \sum^{\lfloor nt_1\rfloor-1}_{j=0}\int_{0}^{1 }  |  (\lfloor nt_2\rfloor-j)^{\alpha}-(\lfloor nt_2\rfloor -j-y)^{\alpha} | \\
 & \qquad  \times |(\lfloor nt_1\rfloor -j)^{\alpha}-(\lfloor nt_1\rfloor-j-y)^{\alpha} |
 dy\\
 &  = \sum^{\lfloor nt_1\rfloor }_{k=1}  \int_0^1   |  (\lfloor nt_2\rfloor- \lfloor nt_1\rfloor+k)^{\alpha}-(\lfloor nt_2\rfloor -\lfloor nt_1\rfloor+k-y)^{\alpha} |  \\
 & \qquad  \times |k^{\alpha}-(k-y )^{\alpha} |dy,
\end{align*}
which converges to zero as $n\to \infty$. This completes the proof of the proposition.
\end{proof}

\medskip
\noindent
{\bf Remark}. 
Notice that if we replace $\eta_n(t)$  by $t$ in the proof of Proposition \ref{first_term}, then instead of \eqref{problem}, we obtain
\[
E[(N^{n,1}_{t})^2]  
= \sum^{\lfloor nt\rfloor}_{k=0}\int^{1 \wedge (nt-k)}_{0} \left[(nt-k)^{\alpha}-(nt-k-y)^{\alpha}\right]^2\, dy,
\]
and this term does not seem to converge as $n$ tends to infinity due to the fluctuation of $nt - \lfloor nt \rfloor$ between $0$ and $1$ as $n$ tends to infinity.

\subsection{Decomposition of the modified normalized error $\wY^n_t$}
Because of the particular behavior of  $\wA^{n}_t$, we will consider the  decomposition
\[
\wY^n_t= \wA^{n}_t + \wY^{n,1}_t.
\]
From Proposition \ref{propY} we deduce that
\[
\lim _{n\rightarrow \infty} \sup_{t\in [0,T]} E[| \wY^{n,1}_t -Y^{n,1}_t|^2] =0,
\]
where  $Y^{n,1}_t$ has been introduced in  \eqref{decom}.

The process 
   $\wY^{n,1}_t= \wY^n_t -\wA^{n}_t$ satisfies the equation
\begin{equation} \label{dnr1}
\wY^{n,1}_t =  \widetilde{C}^n_t +\int_0^t(t-s)^{\alpha}\sigma'(X_s) (\wY^{n,1}_s + \wA^n_s)\,dW_s.
\end{equation}
With the notation
\begin{align}   \nonumber
\wA^{n,1}_t& =\int_0^t(t-s)^{\alpha}\sigma'(X_s)  \wA^n_s\,dW_s \\
&= n^{\alpha+\frac 12} \int_0^t(t-s)^{\alpha}\sigma'(X_s)  \sigma(X_s) \left( \int_0^s \psi_{n,1}(u,s)dW_u \right)dW_s, \label{dnr3}
\end{align}
we can write equation \eqref{dnr1} as
\begin{equation} \label{dnr2}
\wY^{n,1}_t = \wA^{n,1}_t  + \widetilde{C}^n_t +\int_0^t(t-s)^{\alpha}\sigma'(X_s) \wY^{n,1}_s\,dW_s.
\end{equation}

 \subsection{Asymptotic behavior of $\wA^{n,1}_t+\widetilde{C}^n_t$}  \label{sec3}
 In this section we will study the limit as $n$ tends to infinity of the   random process  $\{\wA^{n,1}_t+\wC^n_t, t\in [0,T]\}$ where $\wA^{n,1}_t$
 has been introduced in  \eqref{dnr3} and $\wC^n_t$ is defined in   \eqref{C2}.
 The main result is the following proposition.
 
 \begin{proposition} \label{prop1}
 The finite-dimensional distributions of the process $\{\wA^{n,1}_t+\wC^n_t, t\in [0,T]\}$ converge in law to those of
 \[
\left\{ \kappa_2  \int^{t}_0(t-s)^{\alpha}(\sigma'\sigma)(X_s)dB_s , t \in [0,T] \right\},
 \]
 where $B=\{ B _t, t\in [0,T] \}$ is a   standard Brownian motion independent of $W$ and $\kappa_2$  is the constant defined in \eqref{kappa2}.
 \end{proposition}
 
\begin{proof}   
Fix points  $t_1, \dots, t_Q \in [0,T]$ and  real numbers $\rho_i$,$1\le i \le Q$. Define the martingale 
\begin{align*}
M^{(n)}_\tau &=  \sum_{i=1}^Q \rho_i \int_0^{\tau \wedge t_i}(t_i-s)^{\alpha}\sigma'(X_s)  \wA^{n}_s dW_s \\
& \qquad +
 n^{\alpha+\frac 12}  \sum_{i=1}^Q \rho_i \int_0^{\tau \wedge t_i}(t_i-s)^{\alpha}\sigma'(X_s)  (X^n_{\eta_n(s)} - X^n_s) dW_s \\
 &=
n^{\alpha+\frac 12}  \sum_{i=1}^Q \rho_i \int_0^{\tau \wedge t_i}(t_i-s)^{\alpha}\sigma'(X_s)
\Theta^n_s dW_s,
 \end{align*}
 where
 \begin{equation} \label{theta1}
 \Theta^n_s :=\sigma(X_s)  
 \left( \int_0^s \psi_{n,1}(u,s) dW_u \right)
 +  X^n_{\eta_n(s)} - X^n_s.
 \end{equation}
 Then, in order to prove the proposition, it suffices to prove that,    as $n$ tends to infinity,  $ M^{(n)}_\tau$ converges in distribution to
 \[
\kappa_2 ^2   \sum_{i=1}^Q \rho_i   \int^{\tau \wedge t_i}_0(t_i-s)^{\alpha}(\sigma'\sigma)(X_s)dB_s,
\]
where   $B=\{ B _t, t\in [0,T]\}$ is a   standard Brownian motion independent of $W$.
  By the asymptotic version of Knight's theorem 
(as presented in \cite[Chapter~XIII,~Theorem~2.3]{RY}),  it suffices to show that the following two conditions hold:

\medskip
\noindent
{\bf (C1)} As $n$ tends to infinity,  the quadratic variation $\langle M^{(n)} \rangle_\tau$ 
given by
\begin{equation} \label{C1}
\langle M^{(n)} \rangle_\tau = 
n^{2\alpha+1}    \int_0^{\tau } \left|  \sum_{i=1}^Q \rho_i  \mathbf{1}_{[0,t_i]}(\tau) (t_i-s)^{\alpha}\sigma'(X_s)
 \Theta^n_s \right| ^2 ds,
\end{equation}
converges in $L^2$,  uniformly in $\tau  \in [0,T]$, to
\[
\kappa_2^2 \int_0^\tau \left|  \sum_{i=1}^Q \rho_i  \mathbf{1}_{[0,t_i]}(\tau) (t_i-s)^{\alpha}(\sigma'\sigma)(X_s)  \right|^2 ds.
\]
 
\medskip
\noindent
{\bf (C2)} As $n$ tends to infinity,   the joint variation
\[
\langle M^{(n)}, W \rangle_\tau
= n^{\alpha+\frac 12}    \int_0^{\tau }   \sum_{i=1}^Q \rho_i  \mathbf{1}_{[0,t_i]}(\tau) (t_i-s)^{\alpha}\sigma'(X_s)
 \Theta^n_s  ds
 \]
converges to zero in $L^2$, uniformly in $\tau  \in [0,T]$.

By linearity, to show properties  \textbf{(C1)}  and \textbf{(C2)}, it suffices to prove that   the following two properties hold true:

\medskip
\noindent
{\bf (C3)}
 For any  $t_1,t_2   \in [0,T]$ 
 \begin{align*}
 &\lim_{n\rightarrow \infty}
 n^{2\alpha+1}    \int_0^{\tau \wedge t_1 \wedge t_2  } (t_1-s)^{\alpha}  (t_2-s)^{\alpha} (\sigma'(X_s)
 \Theta^n_s )^2 ds \\
 & \qquad = \kappa_2^2 \int_0^{\tau \wedge t_1 \wedge t_2  } (t_1-s)^{\alpha}  (t_2-s)^{\alpha} 
(\sigma'\sigma)^2(X_s) ds,
\end{align*}
where the convergence holds in $L^2$, uniformly in $\tau  \in [0,T]$.

\medskip
\noindent
{\bf (C4)}
 For any  $t   \in [0,T]$ 
\[
 \lim_{n\rightarrow \infty} n^{\alpha+\frac 12}      \int_0^{\tau \wedge t  } (t-s)^{\alpha}    \sigma'(X_s)
 \Theta^n_s  ds=0,
 \]
 where the convergence holds in $L^2$, uniformly in $\tau  \in [0,T]$. 

\medskip
\noindent
{\it Proof of {\bf (C3)}}:  We will make use of the   notation
\begin{equation} \label{psi2}
\psi_{n,2}(u,s)=(s-\eta_n(u))^{\alpha}-(\eta_n(s)-\eta_n(u))^{\alpha}.
\end{equation}
Recall that   $\Theta^n_s$ has been defined in \eqref{theta1}.
Then, we can write
\begin{align*}
\Theta^n_s&= \sigma(X_s)  
  \int_0^s \psi_{n,1}(u,s) dW_u 
 +  X^n_{\eta_n(s)} - X^n_s   \\
 &= \sigma(X_s)  
  \int_0^s \psi_{n,1}(u,s) dW_u 
 +  \int^s_{\eta_n(s)}\left(s-\eta_n(u)\right)^{\alpha}\sigma(X^n_{\eta_n(u)})\, dW_u \\
 &  \qquad + \int^{\eta_n(s)}_0 \psi_{n,2}(u,s)\, \sigma(X^n_{\eta_n(u)})\, dW_u.
 \end{align*}
 The first step in the proof of {\bf (C3)} consists on replacing $\sigma(X^n_{\eta_n(u)})$ by
 $\sigma(X_{\eta_n(u)})$ in the expression of $\Theta^n_s$. That is, if we define
 \begin{align} \notag
\widehat{\Theta}^n_s&= \sigma(X_s)  
  \int_0^s \psi_{n,1}(u,s) dW_u 
 +  X^n_{\eta_n(s)} - X^n_s   \\ \notag
 &= \sigma(X_s)  
  \int_0^s \psi_{n,1}(u,s) dW_u 
 +  \int^s_{\eta_n(s)}\left(s-\eta_n(u)\right)^{\alpha}\sigma(X_{\eta_n(u)})\, dW_u \\  \label{hattheta}
 &  \qquad + \int^{\eta_n(s)}_0 \psi_{n,2}(u,s)\, \sigma(X_{\eta_n(u)})\, dW_u,
 \end{align}
 then
 \begin{equation} \label{est1}
 \lim_{n\rightarrow \infty} \sup_{s\in [0,T]}n^{2\alpha+1}  E[ |\widehat \Theta^n_s - \Theta^n_s|^2]=0.
 \end{equation}
 It is easy to check that \eqref{est1} will guarantee that  if {\bf (C3)} holds for $\widehat{\Theta}^n_s$, it also holds for $\Theta^n_s$.
 The convergence \eqref{est1} is a consequence of the limits
  \begin{equation} \label{est2}
 \lim_{n\rightarrow \infty} \sup_{s\in [0,T]}n^{2\alpha+1}  \int_{\eta_n(s)}^s (s-\eta_n(u))^{2\alpha} E[ | \sigma(X^n_{\eta_n(u)}) - \sigma(X_{\eta_n(u)})|^2] du =0
 \end{equation}
 and
  \begin{equation} \label{est3}
   \lim_{n\rightarrow \infty} \sup_{s\in [0,T]}n^{2\alpha+1}  \int_0^{\eta_n(s)} \psi_{n,2}^2(u,s)  E[ | \sigma(X^n_{\eta_n(u)}) - \sigma(X_{\eta_n(u)})|^2] du =0.
 \end{equation}
  The convergences \eqref{est2} and \eqref{est3} can be proved using the Lipschitz property of $\sigma$, the estimate  \eqref{rate} and Lemma  \ref{lema1}. Next, we make the following decomposition of $(\widehat{\Theta}^n_s)^2$:
\[
  (\widehat{\Theta}^n_s)^2= \sum_{j=1}^6  \Psi^{n,j}_s,
 \]
  where
  \[
  \Psi^{n,1}_s =\sigma^2(X_s)  
 \left( \int_0^s \psi_{n,1}(u,s) dW_u \right)^2,
 \]
 \[
  \Psi^{n,2}_s= \left(  \int^s_{\eta_n(s)}\left(s-\eta_n(u)\right)^{\alpha}\sigma(X_{\eta_n(u)})\, dW_u \right)^2,
  \]
  \[
    \Psi^{n,3}_s =\left( \int^{\eta_n(s)}_0 \psi_{n,2}(u,s)\, \sigma(X_{\eta_n(u)})\, dW_u \right)^2,
    \]
  \[
      \Psi^{n,4}_s = 2\sigma(X_s)  
 \left( \int_0^s \psi_{n,1}(u,s) dW_u \right)\left(  \int^s_{\eta_n(s)}\left(s-\eta_n(u)\right)^{\alpha}\sigma(X_{\eta_n(u)})\, dW_u \right),
 \]
   \[
      \Psi^{n,5}_s = 2\sigma(X_s)  
 \left( \int_0^s \psi_{n,1}(u,s) dW_u \right) \left(\int^{\eta_n(s)}_0 \psi_{n,2}(u,s)\, \sigma(X_{\eta_n(u)})\, dW_u \right)
 \]
 and
    \[
      \Psi^{n,6}_s =2\left(  \int^s_{\eta_n(s)}\left(s-\eta_n(u)\right)^{\alpha}\sigma(X_{\eta_n(u)})\, dW_u \right)
      \left(\int^{\eta_n(s)}_0 \psi_{n,2}(u,s)\, \sigma(X_{\eta_n(u)})\, dW_u \right).
      \]
  Then, condition {\bf (C3)} will be a consequence of the following limits in $L^2$,  uniformly in $\tau  \in [0,T]$,
  \begin{align}    \nonumber
  & \lim_{n\rightarrow \infty}
  n^{2\alpha+1}    \int_0^{\tau \wedge t_1 \wedge t_2  } (t_1-s)^{\alpha}  (t_2-s)^{\alpha} (\sigma'(X_s))^2  \Psi^{n,1}_s ds\\
 &\qquad 
  =\kappa_3  \int_0^{\tau \wedge t_1 \wedge t_2  } (t_1-s)^{\alpha}  (t_2-s)^{\alpha} (\sigma'\sigma)^2(X_s)ds,
    \label{limit1}
   \end{align}
     \begin{align}    \nonumber
  & \lim_{n\rightarrow \infty}
  n^{2\alpha+1}    \int_0^{\tau \wedge t_1 \wedge t_2  } (t_1-s)^{\alpha}  (t_2-s)^{\alpha} (\sigma'(X_s))^2  \Psi^{n,2}_s ds\\
 &\qquad  =
\frac 1{ 2\alpha+2} \int_0^{\tau \wedge t_1 \wedge t_2  } (t_1-s)^{\alpha}  (t_2-s)^{\alpha}(\sigma'\sigma)^2(X_s) ds,
    \label{limit2}
   \end{align}
        \begin{align}    \nonumber
  & \lim_{n\rightarrow \infty}
  n^{2\alpha+1}    \int_0^{\tau \wedge t_1 \wedge t_2  } (t_1-s)^{\alpha}  (t_2-s)^{\alpha} (\sigma'(X_s))^2  \Psi^{n,3}_s ds\\
 &\qquad  =
\kappa_4 \int_0^{\tau \wedge t_1 \wedge t_2  } (t_1-s)^{\alpha}  (t_2-s)^{\alpha}(\sigma'\sigma)^2(X_s) ds,
    \label{limit3}
   \end{align}
       \begin{align}    \nonumber
  & \lim_{n\rightarrow \infty}
  n^{2\alpha+1}    \int_0^{\tau \wedge t_1 \wedge t_2  } (t_1-s)^{\alpha}  (t_2-s)^{\alpha} (\sigma'(X_s))^2  \Psi^{n,4}_s ds\\
 &\qquad  =
\frac {\alpha} {2(\alpha+1)^2}\int_0^{\tau \wedge t_1 \wedge t_2  } (t_1-s)^{\alpha}  (t_2-s)^{\alpha}(\sigma'\sigma)^2(X_s) ds,
    \label{limit3a}
   \end{align}
          \begin{align}    \nonumber
  & \lim_{n\rightarrow \infty}
  n^{2\alpha+1}    \int_0^{\tau \wedge t_1 \wedge t_2  } (t_1-s)^{\alpha}  (t_2-s)^{\alpha} (\sigma'(X_s))^2  \Psi^{n,5}_s ds\\
 &\qquad  =
\kappa_5 \int_0^{\tau \wedge t_1 \wedge t_2  } (t_1-s)^{\alpha}  (t_2-s)^{\alpha}(\sigma'\sigma)^2(X_s) ds,
    \label{limit3b}
   \end{align}
   and
         \begin{equation}   
   \lim_{n\rightarrow \infty}
  n^{2\alpha+1}    \int_0^{\tau \wedge t_1 \wedge t_2  } (t_1-s)^{\alpha}  (t_2-s)^{\alpha} (\sigma'(X_s))^2  \Psi^{n,6}_s ds=0,
    \label{limit4}
       \end{equation}
    where $\kappa_3$, $\kappa_4$ and $\kappa_5$  have been introduced in \eqref{kappa3}, \eqref{kappa4} and \eqref{kappa5}, respectively.

  \medskip
  \noindent
  {\it Proof of \eqref{limit1}}:  To simplify the presentation we write
  \begin{equation} \label{gamma}
 \gamma_s:=  \mathbf{1}_{[0,t_1 \wedge t_2]}(s) (t_1-s)^{\alpha}  (t_2-s)^{\alpha}.
 \end{equation}
  Then,
  \begin{align*}
  R^{n,1}_\tau &:=  n^{2\alpha+1}    \int_0^{\tau \wedge t_1 \wedge t_2  } (t_1-s)^{\alpha}  (t_2-s)^{\alpha} (\sigma'(X_s))^2  \Psi^{n,1}_s ds\\
  &= n^{2\alpha+1}    \int_0^\tau  \gamma_s (\sigma'\sigma)^2(X_s)   \left( \int_0^s \psi_{n,1}(u,s) dW_u \right)^2 ds \\
  & =:   R^{n,2}_\tau + R^{n,3}_\tau,
  \end{align*}
 where
 \[
  R^{n,2}_\tau=  n^{2\alpha+1}    \int_0^\tau  \gamma_s (\sigma'\sigma)^2(X_s)  \left[  \left( \int_0^s \psi_{n,1}(u,s) dW_u \right)^2
  -  \int_0^s \psi^2_{n,1}(u,s) du \right] ds,
  \]
  and
  \[
   R^{n,3}_\tau = n^{2\alpha+1}    \int_0^\tau  \gamma_s (\sigma'\sigma)^2(X_s)     \int_0^s \psi^2_{n,1}(u,s) du ds.
   \]
 We claim that
 \begin{equation} \label{71}
 \lim_{n\rightarrow \infty}  \sup_{\tau\in [0,T]}
 E[|  R^{n,2}_\tau|^2]=0.
 \end{equation}
 
 \medskip
 \noindent
 {\it Proof of \eqref{71}}:
 The proof of \eqref{71} is rather involved and will be done in two steps. First we will show that we have convergence to zero in $L^2$ if we remove the random factor
 $ (\sigma'\sigma)^2(X_s)$ and we integrate on any fixed interval $[\tau_1, \tau_2] \subset [0,\tau]$. Secondly,  we will use the small blocks / big blocks argument  (see, for instance, \cite{CNP}) to show that we have convergence to zero when we include the random weight $ (\sigma'\sigma)^2(X_s)$.
 
  \medskip
 \noindent
 {\it Step (i)}: We claim that for any interval  $[\tau_1, \tau_2] \subset [0,\tau]$
 \begin{equation}  \label{L2}
R^{n,4}:= n^{2\alpha+1}\int_{\tau_1}^{\tau_2}\gamma_s \left[\left( \int^{s}_0 \psi_{n,1}(u,s) dW_u\right)^2-  \int^{s}_0\psi_{n,1}^2(u,s)du \,\right] ds \rightarrow 0,
\end{equation}
where the convergence holds in $L^2$, uniformly in $\tau_1, \tau_2$ and $\tau$. Notice that
\[
\left( \int^{s}_0 \psi_{n,1}(u,s) dW_u\right)^2-  \int^{s}_0\psi_{n,1}^2(u,s)du = I_2\left( \psi_{n,1}^{\otimes 2}(\cdot, s) \mathbf{1}_{[0,s]^2}  \right),
\]
where $I_2$ denotes the double stochastic integral with respect to the Brownian motion $W$.
 Therefore,
 \begin{align*}
 E[|R^{n,4}|^2] &=2n^{4\alpha+2} \int_{\tau_1}^{\tau_2}\int_{\tau_1}^{\tau_2} \gamma_{s_1} \gamma_{s_2}  
 \langle \psi^{\otimes 2}_{n,1}(\cdot, s_1) \mathbf{1}^{\otimes 2}_{[0,s_1]}, \psi^{\otimes 2}_{n,1}(\cdot, s_2) \mathbf{1}^{\otimes 2}_{[0,s_2]}  \rangle_{L^2(\mathbb{R}^2)}  ds_1 ds_2\\
 &=4n^{4\alpha+2} \int_{\tau_1}^{\tau_2} \int_{\tau_1}^{s_2}   \gamma_{s_1} \gamma_{s_2} 
\left( \int_0^{s_1} \psi_{n,1}(u, s_1)  \psi_{n,1}(u, s_2) du   \right)^2ds_1 ds_2.
 \end{align*}
 Fix $\delta>0$ such that $\delta> \frac 1n$. On the set $\{ (s_1,s_2): s_2-s_1 >\delta\}$ we have, by the Mean Value Theorem,
 \[
 | \psi_{n,1}(u, s_2)|= | (s_2- \eta_n(u))^{\alpha} - (s_2-u)^\alpha|  \le  \frac {\alpha}n  \delta^{\alpha-1}.
 \]
 As a consequence, using the inequality \eqref{ineq1}, we obtain
 \begin{align*}
 E[|R^{n,4}|^2]  & \le 
 C n^{2\alpha-1} \delta^{2\alpha-2} \int_{\tau_1}^{\tau_2} \int_{\tau_1}^{s_2}   \gamma_{s_1} \gamma_{s_2}  \mathbf{1}_{\{s_2-s_1 >\delta\}}
 ds_1 ds_2\\
 & \qquad +C \int_{\tau_1}^{\tau_2} \int_{\tau_1}^{s_2}   \gamma_{s_1} \gamma_{s_2}  \mathbf{1}_{\{s_2-s_1 \le \delta\}} ds_1 ds_2 \\
 & \le  C n^{2\alpha-1} \delta^{2\alpha-2} \left(  \int_0^T \gamma_{s}ds \right)^2 +C \int_{0}^{T} \int_{0}^{s_2}   \gamma_{s_1} \gamma_{s_2}  \mathbf{1}_{\{s_2-s_1 \le \delta\}} ds_1 ds_2 .
 \end{align*} 
 This implies
 \[
 \limsup_{n\rightarrow \infty}  \sup_{\tau_1, \tau_2 \le  \tau \in [0,T]}  E[|R^{n,4}|^2] \le C \int_{0}^{T} \int_{0}^{s_2}   \gamma_{s_1} \gamma_{s_2}  \mathbf{1}_{\{s_2-s_1 \le \delta\}} ds_1 ds_2.
 \]
 Taking into account that $\delta>0$ is arbitrary, this completes the proof of \eqref{L2}.

   \medskip
 \noindent
 {\it Step (ii)}:  Fix an integer $M\ge 1$ and consider a partition $0=\tau_0 <\tau_1 < \cdots < \tau_M =\tau$, where
$ \tau_i= \frac iM \tau$ for $i=0,1, \dots, M$. With the notation
\[
\Xi ^{n,1}_s=
\left[\left( \int^{s}_0 \psi_{n,1}(u,s) dW_u\right)^2-  \int^{s}_0\psi_{n,1}^2(u,s)du \,\right] ,
\]
we can write
\begin{align*}
 R^{n,3} _\tau &=
n^{2\alpha+1}\int^{\tau}_0\gamma_s (\sigma'\sigma)^2(X_s) \Xi ^{n,1}_s ds \\
&=n^{2\alpha+1} \sum_{i=0} ^{M-1}  \int_{\tau_i} ^{\tau_{i+1}} \gamma_s (\sigma'\sigma)^2(X_s) \Xi^{n,1}_s ds \\
&=:   R^{n,M,1} _\tau  + R^{n,M,2} _\tau,
\end{align*}
  where
  \[
  R^{n,M,1} _\tau  =
n^{2\alpha+1} \sum_{i=0} ^{M-1}  (\sigma'\sigma)^2(X_{\tau_i})  \int_{\tau_i} ^{\tau_{i+1}} \gamma_s  \Xi^{n,1}_s ds 
\]
and
 \[
  R^{n,M,2} _\tau  =
n^{2\alpha+1} \sum_{i=0} ^{M-1}    \int_{\tau_i} ^{\tau_{i+1}} \gamma_s    
\left[   (\sigma'\sigma)^2(X_s) - (\sigma'\sigma)^2(X_{\tau_i} )  \right] \Xi^{n,1}_s ds 
\]
From Step 1, we know that for any fixed $M$
\begin{equation}  \label{L3a}
\lim_{n\rightarrow \infty} \sup_{\tau\in [0,T]} E[|  R^{n,M,1} _\tau|^2]=0.
\end{equation}
Indeed, for each $0\le i \le M-1$, the term  $n^{2\alpha+1} \int_{\tau_i} ^{\tau_{i+1}} \gamma_s  \Xi^{n,1}_s ds $ converges to zero in $L^2$ as $n\to \infty$, uniformly in $i=0, \dots, M-1$ and $\tau \in [0,T]$, and this convergence also holds in $L^p$ for any $p\ge 2$, because $ \Xi^{n,1}_s $ belongs to the second Wiener chaos.

Assuming that $ n\ge M$,  the term    $R^{n,M,2} _\tau $ can be estimated as follows
\[
 \| R^{n,M,2} _\tau  \|_2 \le C \sup_{u,v \in [0,1], |u-v| \le \frac \tau M} \|  (\sigma'\sigma)^2(X_u) - (\sigma'\sigma)^2(X_v)  \|_4
n^{2\alpha+1}  \int_0^\tau  \gamma_s    \| \Xi^{n,1}_s \|_4 ds.
\]
From Lemma  \ref{lema2} we get
\[
\sup_{n\ge 1} n^{2\alpha+1}  \int_0^T  \gamma_s    \| \Xi^{n,1}_s \|_4 ds <\infty.
\]
Therefore,
\begin{equation} \label{L4a}
\lim_{M\rightarrow \infty} \sup_{n\ge M}   \sup_{\tau\in [0,T]} \| R^{n,M,2} _\tau  \|_2=0.
\end{equation}
In conclusion, \eqref{L3a} and \eqref{L4a} imply  \eqref{71}.

\medskip
It remains to compute the limit as $n$ tends to infinity of the term  $R^{n,3}_\tau $.
 First notice that the term     $R^{n,3}_\tau $ can be expressed as follows
 \begin{align*}
 R^{n,3}_\tau &=n^{2\alpha+1}    \int_0^\tau  \gamma_s (\sigma'\sigma)^2(X_s)     \int_0^s   [ (s- \eta_n(u))^\alpha - (s-u)^\alpha]^2 du ds \\
 &=n^{2\alpha+1}   \sum_{j=0}^{ \lfloor n\tau \rfloor} \int_{\frac jn} ^{\frac {j+1}n \wedge \tau}
 \gamma_s(\sigma'\sigma)^2(X_s)  \sum_{i=0}^{ \lfloor ns \rfloor} \int_{\frac in} ^{\frac {i+1}n \wedge s} [ (s-\frac in)^\alpha - (s-u)^\alpha]^2 duds.
 \end{align*}
 The change of variables  $nu=v$ and $ns=x$ yields
\[
 R^{n,3}_\tau = \frac 1n   \sum_{j=0}^{ \lfloor n\tau \rfloor} \int_{j} ^{(j+1) \wedge n\tau}
 \gamma_{\frac xn}(\sigma'\sigma)^2(X_{\frac xn})  \sum_{i=0}^{ \lfloor x \rfloor} \int_{i} ^{(i+1) \wedge x} [ (x-i)^\alpha - (x-v)^\alpha]^2 dvdx.
\]
 With the additional change of variables $x-j=y$ and $v-i=z$, we can write
  \begin{align*}
 R^{n,3}_\tau &= \frac 1n   \sum_{j=0}^{ \lfloor n\tau \rfloor} \int_{0} ^{1 \wedge (n\tau-j)}
 \gamma_{\frac {y+j}n}(\sigma'\sigma)^2(X_{\frac {y+j}n}) \\
 &\qquad \times  \sum_{i=0}^{ \lfloor y+j \rfloor} \int_{0} ^{1 \wedge (y+j-i)} [ (y+j-i)^\alpha - (y+j-i-z)^\alpha]^2 dzdy\\
 &=\frac 1n   \sum_{j=0}^{ \lfloor n\tau \rfloor} \int_{0} ^{1 \wedge (n\tau-j)}
 \gamma_{\frac {y+j}n}(\sigma'\sigma)^2(X_{\frac {y+j}n})  \sum_{k= j- \lfloor y+j \rfloor}^{j} \int_{0} ^{1 \wedge (y+k)} [ (y+k)^\alpha - (y+k-z)^\alpha]^2 dzdy.
 \end{align*}
 Because $y \le 1$,  we have that $\lfloor y+j \rfloor = j$. Therefore, we obtain
 \begin{align*}
 \lim_{n\rightarrow \infty}
  R^{n,3}_\tau &=
   \lim_{n\rightarrow \infty}
   \frac 1n   \sum_{j=0}^{ \lfloor n\tau \rfloor}  
 \gamma_{\frac {j}n}(\sigma'\sigma)^2(X_{\frac {j}n})  \int_0^1 \sum_{k= 0}^{j} \int_{0} ^{1 \wedge (y+k)} [ (y+k)^\alpha - (y+k-z)^\alpha]^2 dzdy \\
 &= \left(\sum_{k= 0}^{\infty}\int_0^1   \int_{0} ^{1 \wedge (y+k)} [ (y+k)^\alpha - (y+k-z)^\alpha]^2 dzdy \right)\\
 & \qquad \times
  \int_0^{\tau \wedge t_1 \wedge t_2  } (t_1-s)^{\alpha}  (t_2-s)^{\alpha}  (\sigma'\sigma)^2(X_s) ds,
   \end{align*}
   where the convergence is in $L^2$, uniformly in $\tau \in [0,T]$.
   This completes the proof of \eqref{limit1}.
   
   \medskip
  \noindent
  {\it Proof of \eqref{limit2}}: 
  With the notation \eqref{gamma}, we have
  \begin{align*}
 Q^n_\tau& :=   n^{2\alpha+1}    \int_0^{\tau \wedge t_1 \wedge t_2  } (t_1-s)^{\alpha}  (t_2-s)^{\alpha} (\sigma'(X_s))^2  \Psi^{n,2}_s ds \\
 &=  n^{2\alpha+1}    \int_0^{\tau} \gamma_s (\sigma'(X_s))^2  \left(  \int^s_{\eta_n(s)}\left(s-\eta_n(u)\right)^{\alpha}\sigma(X_{\eta_n(u)})\, dW_u \right)^2 ds \\
&= n^{2\alpha+1}\int^{\tau}_0\gamma_s (s-\eta_n(s))^{2\alpha}\left( \sigma'(X_s)\right)^2\sigma^2(X_{\eta_n(s)}) \, \left(W_s-W_{\eta_n(s)}\right)^2\, ds\\
&=n^{2\alpha+1}\int^{\tau}_0\gamma_s(s-\eta_n(s))^{2\alpha}\left( \sigma'(X_s)\right)^2\sigma^2(X_{\eta_n(s)}) \, \left[\left(W_s-W_{\eta_n(s)}\right)^2 -(s-\eta_n(s))\right]\,ds\\
& \qquad + n^{2\alpha+1}\int^{\tau}_0 \gamma_s (s-\eta_n(s))^{2\alpha+1}\left(\sigma'(X_s)\right)^2\,\sigma^2(X_{\eta_n(s)}) ds\\
&= :Q^{n,1}_\tau+ Q^{n,2}_\tau.
 \end{align*}
We claim that
\begin{equation} \label{lim6}
\lim_{n\rightarrow \infty}  \sup_{\tau\in [0,T]} E\left[ |Q^{n,1} _\tau|^2 \right]=0.
\end{equation}
In order to show  the convergence \eqref{lim6}, we first replace $\sigma'(X_s)$ by $\sigma'(X_{\eta_n(s)})$ in $Q^{n,1}_\tau$ and then use
the fact that we are dealing with a discrete martingale sequence.
That is, we make the decomposition
\[
Q^{n,1} _\tau =: Q^{n,3} _\tau+ Q^{n,4} _\tau,
\]
where
\begin{align*}
Q^{n,3} _\tau
&= n^{2\alpha+1}\int^{\tau}_0\gamma_s(s-\eta_n(s))^{2\alpha}\left(  (\sigma'(X_s))^2- ( \sigma'(X_{\eta_n(s)}))^2 \right)\sigma^2(X_{\eta_n(s)}) \, \\
& \qquad \times \left[\left(W_s-W_{\eta_n(s)}\right)^2 -(s-\eta_n(s))\right]\,ds
\end{align*}
and
\[
Q^{n.4}_\tau=n^{2\alpha+1}\int^{\tau}_0 \gamma_s(s-\eta_n(s))^{2\alpha}  (\sigma' \sigma)^2(X_{\eta_n(s)}) \, \left[\left(W_s-W_{\eta_n(s)}\right)^2 -(s-\eta_n(s))\right]\,ds.
\]
Then, using Minkowski's inequality, we obtain
\begin{align*}
\| Q^{n,3} _\tau \|_2  &\le n^{2\alpha+1}  \int^{\tau}_0 \gamma_s(s-\eta_n(s))^{2\alpha} \| \left(  (\sigma'(X_s))^2- ( \sigma'(X_{\eta_n(s)}))^2 \right)\sigma^2(X_{\eta_n(s)}) \|_4   \\
& \qquad \times  \|[(W_s-W_{\eta_n(s)})^2 -(s-\eta_n(s))]\|_4 ds.
\end{align*}
Because $\sigma'$ is $\beta$-H\"older continuous and $\sigma$ has linear growth, we can write using the estimates \eqref{sup} and \eqref{rgla}
\[
\| Q^{n,3} _\tau \|_2   \le  Cn^{2\alpha+1-\beta(\alpha+\frac 12)}  \int^{\tau}_0 \gamma_s (s-\eta_n(s))^{2\alpha+1}    ds \le C n^{-\beta(\alpha+\frac 12)},
\]
which shows that
\begin{equation} \label{h13}
\lim_{n\rightarrow \infty}  \sup_{\tau\in [0,T]} E\left[ |Q^{n,3} _\tau|^2 \right]=0.
\end{equation}
 To handle the term $Q^{n,4}_\tau$ we write
\[
 Q^{n,4} _\tau  = n^{2\alpha+1} \sum_{j=0} ^{\lfloor n\tau \rfloor} \xi_{j,n},
 \]
 where
\[
\xi_{j,n} = \int_{\frac jn} ^{\frac {j+1}n \wedge \tau}
 \gamma_s (s-\frac jn)^{2\alpha}  (\sigma'\sigma)^2(X_{\frac jn}) \left[(W_s- W_{\frac jn})^2- (s-\frac jn) \right] ds.
 \]
 The random variables $\xi_{j,n}$ for $j=0, \dots, \lfloor n\tau \rfloor$,  are $\mathcal{F}_{\frac {j+1} n }$-measurable and satisfy
 $E[ \xi_{j,n} | \mathcal{F}_{\frac jn}]=0$. That means, they form a discrete martingale sequence. As a consequence, for $j\not = k$ we have
 $E[\xi_{j,n} \xi_{k,n} ]=0$ and this implies
 \begin{equation} \label{h4}
 E[| Q^{n,4} _\tau|^2]= n^{4\alpha+2} \sum_{j=0} ^{\lfloor n\tau \rfloor}  E[\xi_{j,n}^2].
 \end{equation}
 By Minkowski's inequality
 \begin{equation}  \label{h3}
 \| \xi_{j,n} \|_2  \le  C \int_{\frac jn} ^{\frac {j+1}n \wedge \tau}
 \gamma_s (s-\frac jn)^{2\alpha+1}   ds \le  Cn^{ -2\alpha-1} \int_{\frac jn} ^{\frac {j+1}n \wedge \tau}
 \gamma_s ds.
\end{equation}
Substituting \eqref{h3} into  \eqref{h4}, we obtain
\[
  E[| Q^{n,4} _\tau|^2]  \le C \sum_{j=0} ^{\lfloor n\tau \rfloor}   \left(\int_{\frac jn} ^{\frac {j+1}n \wedge \tau}
\gamma_s ds  \right)^2 \le C \sup_{0\le j \le \lfloor n\tau \rfloor} \left(\int_{\frac jn} ^{\frac {j+1}n \wedge \tau}
\gamma_s ds  \right) \int_0^{\tau}  \gamma_s ds,
\]
 which implies
 \begin{equation} \label{h5}
\lim_{n\rightarrow \infty} \sup_{\tau\in [0,T]}  E\left[ |Q^{n,4} _\tau|^2 \right]=0.
\end{equation}
 Therefore, \eqref{h13} and \eqref{h5} imply \eqref{lim6}.
 
 In order to find the limit of the term $Q^{n,2} _\tau$, we first make the following decomposition
 \[
 Q^{n,2} _\tau =Q^{n,5} _\tau +Q^{n,6} _\tau,
 \]
 where
 \[
 Q^{n,5} _\tau = n^{2\alpha+1}\int^{\tau}_0 \gamma_s (s-\eta_n(s))^{2\alpha+1}[(\sigma'(X_s))^2- (\sigma'(X_{\eta_n(s)}))^2]\,\sigma^2(X_{\eta_n(s)}) ds,
 \]
 and
  \[
 Q^{n,6} _\tau = n^{2\alpha+1}\int^{\tau}_0 \gamma_s (s-\eta_n(s))^{2\alpha+1}
 (\sigma' \sigma)^2 (X_{\eta_n(s)})ds.
 \]
The $\beta$-H\"older continuity of $\sigma'$  together with the estimates \eqref{sup} and \eqref{rgla} imply
  \begin{equation} \label{h6}
\lim_{n\rightarrow \infty}  \sup_{\tau\in [0,T]} E\left[ |Q^{n,5} _\tau|^2 \right]=0.
\end{equation}
 Then, we make a further decomposition of the term $Q^{6,n}_\tau$ as follows
 \[
 Q^{n,6} _\tau =Q^{n,7} _\tau +Q^{n,8} _\tau,
 \]
 where
 \[
 Q^{n,7} _\tau = n^{2\alpha+1}\int^{\tau}_0[ \gamma_s- \gamma_{\eta_n(s)}](s-\eta_n(s))^{2\alpha+1}
 (\sigma' \sigma)^2 (X_{\eta_n(s)})ds,
 \]
 and
  \[
 Q^{n,8} _\tau = n^{2\alpha+1}\int^{\tau}_0 \gamma_{\eta_n(s)}(s-\eta_n(s))^{2\alpha+1}
 (\sigma' \sigma)^2 (X_{\eta_n(s)})ds.
 \]
 We have, using \eqref{ineq1}
\[
 \|  Q^{n,7} _\tau \|_2   \le C  \int^{t_1\wedge t_2}_0|(t_1-s)^{\alpha}(t_2-s)^\alpha- (t_1-\eta_n(s))^{\alpha} (t_2-\eta_n(s))^{\alpha}  |ds,    
 \]
 which implies
 \begin{equation}
 \label{h7}
 \lim_{n\rightarrow \infty}  \sup_{\tau\in [0,T]} \|  Q^{n,7} _\tau \|_2=0.
\end{equation}
 Moreover, the term $Q^{n,8} _\tau$ can be expressed as follows
 \begin{align*}
 Q^{n,8} _\tau &=n^{2\alpha+1} \sum_{j=0} ^{\lfloor n\tau \rfloor}
 \gamma(\frac jn) (\sigma' \sigma)^2(X_{\frac jn})
  \int_{\frac jn} ^{\frac {j+1} n \wedge \tau} (s-\frac jn)^{2\alpha+1}ds\\
  &=\frac { n^{2\alpha+1}}{2\alpha+2}  \sum_{j=0} ^{\lfloor n\tau \rfloor}
 \gamma(\frac jn) (\sigma' \sigma)^2(X_{\frac jn}) \left( \frac {j+1}n \wedge \tau -\frac jn\right)^{2\alpha+2}\\
&= \frac 1{(2\alpha+2)n} \sum_{j=0} ^{\lfloor n\tau \rfloor-1} \gamma(\frac jn) (\sigma' \sigma)^2(X_{\frac jn})\\
&\qquad 
 +\frac { n^{2\alpha+1}}{2\alpha+2} 
 (\sigma'\sigma)^2 (X_{ \frac { \lfloor n\tau \rfloor}n} ) \gamma( \frac { \lfloor n\tau \rfloor}n)   \left(\tau-  \frac { \lfloor n\tau \rfloor}n  \right)^{2\alpha+2}.
 \end{align*}
 Notice that
\[
\gamma( \frac { \lfloor n\tau \rfloor}n) \left(\tau-  \frac { \lfloor n\tau \rfloor}n  \right)^{2\alpha+2} =  
\mathbf{1} _{\{ \frac {\lfloor n\tau \rfloor}n  \le t_1\wedge t_2 \}} \left( t_1- \frac { \lfloor n\tau \rfloor}n \right)^{\alpha}\left( t_2- \frac { \lfloor n\tau \rfloor}n \right)^{\alpha} \left(\tau-  \frac { \lfloor n\tau \rfloor}n  \right)^{2\alpha+2}.
\]
If $n$ is large enough, this quantity vanishes   if $\tau > t_1 \wedge t_2$,  is bounded by $ Cn^{-2\alpha-2}$ if   $\tau < t_1 \wedge t_2$,
is bounded by  $n^{-\alpha -2} $ if $\tau=t_1 <t_2$ and, finally it is bounded by  $n^{-2}$ if $\tau=t_1=t_2$.
 As a consequence,
 \begin{align} \nonumber
 \lim_{n\rightarrow \infty}  Q^{n,8} _\tau& = \frac 1{2\alpha+2}  \lim_{n\rightarrow \infty}  
 \frac 1n \sum_{j=0} ^{\lfloor n\tau \rfloor-1} \gamma(\frac jn) (\sigma' \sigma)^2(X_{\frac jn})\\ \nonumber
 &= \frac 1{2\alpha+2}  \lim_{n\rightarrow \infty}   \int_0^{ \frac {\lfloor n\tau \rfloor}n} \gamma(\eta_n(s)) (\sigma' \sigma)^2 ( X_{\eta_n(s)}) ds\\
 &=\frac 1{2\alpha+2}    \int_0^{  \tau\wedge t_1 \wedge t_2} (t_1- s)^{\alpha} (t_2-s)^\alpha (\sigma' \sigma)^2 ( X_s) ds,  \label{h8}
 \end{align}
 where the convergence is in $L^2$, uniformly in $\tau \in [0,T]$. The limit in the last equality of the above equation can be proved by arguments similar to those used before based on the estimate \eqref{ineq1} and  the H\"older continuity of $\sigma'$.  From \eqref{h6}, \eqref{h7} and \eqref{h8} we deduce that
 \[
  \lim_{n\rightarrow \infty}  Q^{n,2} _\tau= \frac 1{2\alpha+2}    \int_0^{  \tau\wedge t_1 \wedge t_2} (t_1- s)^{\alpha} (t_2-s)^\alpha (\sigma' \sigma)^2 ( X_s) ds, 
  \]
   where the convergence is in $L^2$, uniformly in $\tau \in [0,T]$.  Together with \eqref{lim6}, this completes the proof of  \eqref{limit2}.

\medskip \noindent
{\it Proof of  \eqref{limit3}:}
We have
\[
 N^n_\tau: =n^{2\alpha+1}\int^{\tau}_0 \gamma_s \left(\sigma'(X_s)\right)^2  \left( \int^{\eta_n(s)}_0\psi_{n,2}(u,s) \sigma(X_{\eta_n(u)})\, dW_u\right)^2\, ds.
\]
First we are going to show that 
\[
N^n_\tau= N^{n,1} _\tau + N^{n,2} _\tau,
\]
where
\[
N^{n,1} _\tau =n^{2\alpha+1}\int^{\tau}_0\gamma_s (\sigma'\sigma)^2(X_{\eta_n(s)})  \left( \int^{\eta_n(s)}_0\psi_{n,2}(u,s)  dW_u\right)^2\, ds
\]
and 
\begin{equation}
\label{k7}
\lim_{n\rightarrow \infty}  \sup_{\tau\in [0,T} E[| N^{n,2} _\tau|^2]=0.
\end{equation}

 Because 
 $\sigma( X_{\eta_n(u)})$ lies inside the stochastic integral with respect to $dW_u$ and 
  $\sigma( X_{\eta_n(s)})$ is not adapted to $\mathcal{F}_u$, to replace the  term   $\sigma( X_{\eta_n(u)})$ by    $\sigma( X_{\eta_n(s)})$  requires a delicate argument. The idea is to show that if  we consider the stochastic integral in an interval of the form $[0, \eta_n(s)- \delta]$, where $\delta>0$, then the corresponding term tends to zero in $L^1$ and the integral over $[\eta_n(s)- \delta, \eta_n(s)]$ can be made arbitrarily small in $L^2$-norm if $\delta $ is small.
 
 Fix $\delta >0$. By Lemma \ref{lema1} and Burkholder and Minkowski's inequalities, we obtain
 \begin{align} \nonumber
 &n^{2\alpha+1}  \left \| \int_{0}^{(\delta+\frac 1n) \wedge \tau}\gamma _s\left(\sigma'(X_s)\right)^2  \left( \int^{\eta_n(s)}_0\psi_{n,2}(u,s) \sigma (X_{\eta_n(u)})\, dW_u\right)^2\, ds \right \|_2 \\ \label{k1}
 & \le C n^{2\alpha+1}  \int_0^{(\delta+\frac 1n) \wedge \tau }\gamma_s  \int^{\eta_n(s)}_0\psi^2_{n,2}(u,s) du ds \le C (\delta + \frac 1n) ^{2\alpha+1}.
 \end{align}
 Define
 \begin{align*}
 K^{n,1}_\tau  &=n^{2\alpha+1}  \int^{\tau}_{(\delta+\frac 1n) \wedge \tau } \gamma_s \left(\sigma'(X_s)\right)^2 \\
 &  \times   \left[ \left( \int_{0}^{\eta_n(s)}\psi_{n,2}(u,s) \sigma(X_{\eta_n(u)})\, dW_u\right)^2
 - \left( \int_{\eta_n(s)-\delta}^{\eta_n(s)}\psi_{n,2}(u,s) \sigma(X_{\eta_n(u)})\, dW_u\right)^2 \right]
 \, ds.
 \end{align*}
 Then,  using Burkholder's inequality we can estimate the $L^2$-norm of $K^{n,1}_\tau$ as follows
 \begin{align*}
 \| K^{n,1} _\tau \|_2 & \le  Cn^{2\alpha+1}\int^{\tau}_{(\delta+\frac 1n) \wedge \tau } \gamma_s  \left \| \int_0^{\eta_n(s)-\delta}\psi_{n,2}(u,s) \sigma(X_{\eta_n(u)})\, dW_u \right\|_4 \\
 & \qquad  \times \left \| \int_0^{\eta_n(s)}\psi_{n,2}(u,s)\sigma(X_{\eta_n(u)}) dW_u 
 + \int_{\eta_n(s)-\delta}^{\eta_n(s)}\psi_{n,2}(u,s)\sigma(X_{\eta_n(u)}) dW_u   \right\|_4 ds \\
 & \le C n^{2\alpha+1} \int^{\tau}_{(\delta+\frac 1n) \wedge \tau } \gamma_s  \left( \int_0^{\eta_n(s)-\delta} \psi_{n,2}^2(u,s) du \right)^{\frac 12}
 \left( \int_0^{\eta_n(s)} \psi_{n,2}^2(u,s) du \right)^{\frac 12}ds \\
 &  \le C  n^{\alpha+\frac 12} \int^{\tau}_{(\delta+\frac 1n) \wedge \tau } \gamma_s \left( \int_0^{\eta_n(s)-\delta} \psi_{n,2}^2(u,s) du \right)^{\frac 12}
ds,
 \end{align*}
 where in the last inequality we have used Lemma \ref{lema1}. Finally, applying  the Mean Value Theorem, we obtain
 \begin{align} \nonumber
  \| K^{n,1} _\tau \|_2  & \le  Cn^{\alpha+ \frac 12} \int^{\tau}_{(\delta+\frac 1n) \wedge \tau }\gamma_s  \left( \int^{\eta_n(s)-\delta}_0 (s-\eta_n(s))^2 ((\eta_n(s)- \eta_n(u))^{2\alpha-2} du \right)^{\frac 12} ds\\ \label{k2}
 & \le C \delta^{\alpha-1} n^{\alpha-\frac 12}.
 \end{align}
 Consider now the term
 \begin{align*}
 K^{n,2}_\tau  &=n^{2\alpha+1}  \int^{\tau}_{(\delta+\frac 1n) \wedge \tau } \gamma_s \left(\sigma'(X_s)\right)^2 \\
 &  \times   \left[ \left( \int_{\eta_n(s)-\delta}^{\eta_n(s)}\psi_{n,2}(u,s) \sigma(X_{\eta_n(u)})\, dW_u\right)^2
 - \left( \int_{\eta_n(s)-\delta}^{\eta_n(s)}\psi_{n,2}(u,s) \sigma(X_{\eta_n(s)-\delta})\, dW_u\right)^2 \right]
 \, ds.
 \end{align*}
 We can estimate the $L^2$-norm of $K^{n,2}_\tau$ as follows
 \begin{align}  \nonumber
 \| K^{n,2} _\tau \|_2 & \le n^{2\alpha+1}\int^{\tau}_{(\delta+\frac 1n) \wedge \tau }\gamma_s   \left \| \int_{\eta_n(s)-\delta}^{\eta_n(s)}\psi_{n,2}(u,s) [\sigma(X_{\eta_n(u)}) - \sigma (X_{\eta_n(s)-\delta})]\, dW_u \right\|_4 \\ \nonumber
 & \times \left \| \int_{\eta_n(s)-\delta}^{\eta_n(s)}\psi_{n,2}(u,s) [\sigma(X_{\eta_n(u)}) + \sigma (X_{\eta_n(s)-\delta})]\, dW_u \right\|_4 ds \\ \label{k3}
 & \le C n^{2\alpha+1} \delta^{\alpha+\frac 12}  \int^{T}_{0}\gamma_s \int_0^{\eta_n(s)} \psi_{n,2}^2(u,s) duds \le C \delta^{\alpha+\frac 12},
 \end{align}
 where in the first inequality we have used  the Lipschitz property of $\sigma$, inequality \eqref{rgla}, together with Burkholder's inequality for the stochastic integral and in the second inequality we have used Lemma \ref{lema1}.
 Define now
  \begin{align*}
 K^{n,3}_\tau  &=n^{2\alpha+1}  \int^{\tau}_{(\delta+\frac 1n) \wedge \tau }\gamma_s [ (\sigma'(X_s) \sigma( X_{\eta_n(s) -\delta}))^2
 - ( \sigma'\sigma)^2(X_{\eta_n(s)})]  \left( \int_{\eta_n(s)-\delta}^{\eta_n(s)}\psi_{n,2}(u,s) dW_u \right)^2 ds
 \end{align*}
 By the same arguments as before, we can write
 \begin{equation} \label{k4}
 \|  K^{n,3}_\tau \|_2 \le C \left( n^{ -\beta(\alpha+\frac 12)} + \delta^{\alpha+\frac 12} \right).
 \end{equation}
 The term
 \begin{align*}
 K^{n,4} _\tau &:=n^{2\alpha+1}  \int^{\tau}_{(\delta+\frac 1n) \wedge \tau }\gamma_s
 ( \sigma'\sigma)^2(X_{\eta_n(s)}) \\
 & \qquad \times  \left[  \left( \int_{\eta_n(s)-\delta}^{\eta_n(s)}\psi_{n,2}(u,s) dW_u \right)^2 
 -\left( \int_0^{\eta_n(s)}\psi_{n,2}(u,s) dW_u \right)^2 \right]ds
 \end{align*}
 can be handled as the term  $K^{n,1} _\tau$ and we get
 \begin{equation} \label{k5}
 \|  K^{n,4} _\tau \|_2 \le C \delta^{\alpha-1} n^{\alpha-\frac 12}.
 \end{equation}
 Finally, as in the proof of \eqref{k1}, we can write
 \begin{equation} \label{k6}
  n^{2\alpha+1} \left \|   \int_0^ {(\delta+\frac 1n)\wedge \tau} \gamma_s
 ( \sigma'\sigma)^2(X_{\eta_n(s)})  
 \left( \int_0^{\eta_n(s)}\psi_{n,2}(u,s) dW_u \right)^2 ds  \right\|_2 \le C (\delta +\frac 1n) ^{2\alpha+1}.
 \end{equation}
 In conclusion, form estimates \eqref{k1}, \eqref{k2}, \eqref{k3}, \eqref{k4}, \eqref{k5} and  \eqref{k6}, by taking first the limit as $n$ tends to infinity and later taking $\delta \downarrow 0$, we obtain \eqref{k7}.
 
 Next we will make a decomposition of the term $N^{n,1} _\tau$ into two parts:
 \[
 N^{n,1}_\tau= N^{n,3} _\tau + N^{n,4} _\tau,
 \]
 where
\[ 
N^{n,3} _\tau=
n^{2\alpha+1}\int^{\tau}_0 \gamma_s (\sigma'\sigma)^2(X_{\eta_n(s)}))\left[\left( \int^{\eta_n(s)}_0 \psi_{n,2}(u,s) dW_u\right)^2-  \int^{\eta_n(s)}_0\psi_{n,2}^2(u,s)du \,\right] ds
\]
 and 
 \[
 N^{n,4} _\tau =n^{2\alpha+1}\int^{\tau}_0\gamma_s (\sigma'\sigma)^2(X_{\eta_n(s)})\int^{\eta_n(s)}_0 \psi_{n,2}^2(u,s)du \,  ds.
 \]
Using the same arguments as in the proof of \eqref{71} we can show that
 \begin{equation} \label{L1}
 \lim_{n\rightarrow \infty}  \sup_{\tau\in [0,T]}E[| N^{n,3} _\tau|^2]=0.
 \end{equation}

\medskip
It  only remains to show that $N^{n,4} _\tau$ converges in $L^2$, uniformly in $\tau \in [0,T]$,  as $n\to \infty$ to
  \[
\kappa_3 \int_0^{\tau \wedge t_1 \wedge t_2  } (t_1-s)^{\alpha}  (t_2-s)^{\alpha}(\sigma'\sigma)^2(X_s) ds,
  \]
  where
  \[
  \kappa_3 = \sum_{k=1} ^\infty  \int_0^1 [(x+k) ^\alpha -k^\alpha]^2 dx.
  \]
We can write
  \begin{align*}
N^{n,4} _\tau &= n^{2\alpha+1}  \sum_{j=0} ^{\lfloor n\tau \rfloor}  \int_{\frac j n } ^{ \frac {j+1} n \wedge  \tau}  \gamma_s (\sigma'\sigma)^2(X_{\frac j n})  \int_0^{\frac jn }
  [ (s-\eta_n(u))^\alpha - ( \frac jn -\eta_n(u) )^{\alpha} ] ^2 du ds \\
  & = n^{2\alpha}  \sum_{j=0} ^{\lfloor n\tau \rfloor}  (\sigma'\sigma)^2(X_{\frac j n})   \int_{\frac j n } ^{ \frac {j+1} n \wedge  \tau}  \gamma_s  \sum_{i=0} ^{j-1}   
  [ (s- \frac in )^\alpha - ( \frac {j-i}n )^{\alpha} ] ^2 ds.
  \end{align*} 
  Making the  change of indices $j-i =k$ and the change of variables $s=\frac y n$ and $y-j=x$, we obtain
  \begin{align*}
     N^{n,4} _\tau &=  n^{2\alpha}  \sum_{j=0} ^{\lfloor n\tau \rfloor}  (\sigma'\sigma)^2(X_{\frac j n})   \int_{\frac j n } ^{ \frac {j+1} n \wedge  \tau}  \gamma_s  \sum_{k=1} ^{j}   
  [ (s- \frac jn  + \frac kn )^\alpha - ( \frac {k}n )^{\alpha} ] ^2 ds \\
  &= n^{-1}  \sum_{j=0} ^{\lfloor n\tau \rfloor}  (\sigma'\sigma)^2(X_{\frac j n})   \int_{j   } ^{  (j+1) \wedge  n \tau}   \gamma_{\frac yn}  \sum_{k=1} ^{j}   
  [ (y- j  + k)^\alpha -  k^{\alpha} ] ^2 dy\\
  &= n^{-1}  \sum_{j=0} ^{\lfloor n\tau \rfloor}  (\sigma'\sigma)^2(X_{\frac j n})   \int_{0 } ^{ 1\wedge(n \tau -j )} \gamma_{\frac {x+j}n}  \sum_{k=1} ^{j}   
  [ (x + k)^\alpha -  k^{\alpha} ] ^2 dx.
  \end{align*}
  We have  $ \lim_{n\rightarrow \infty} \sup_{\tau\in [0,T]} E[ | N^{n,5} _\tau|^2]=0$, where
  \[
 N^{n,5} _\tau:= 
  \frac 1n  \sum_{j=0} ^{\lfloor n\tau \rfloor}  (\sigma'\sigma)^2(X_{\frac j n})   \int_{0 } ^{ 1\wedge(n \tau -j )}  \gamma_{\frac {x+j}n}   \sum_{k=j+1} ^{\infty}   
  [ (x + k)^\alpha -  k^{\alpha} ] ^2 dx=0.
  \]
  Indeed, using  the Mean Value Theorem, we can write
  \[
  \|    N^{n,5} _\tau\|_2 \le   \alpha ^2 \sup_{0\le s\le \tau}   \| (\sigma'\sigma)^2(X_s) \|_2
 \frac 1n    \sum_{j=0}^{\lfloor n\tau \rfloor}     \int_{0 } ^{ 1\wedge(n \tau -j )}  \gamma_{\frac {x+j}n}  dx
 \sum_{k=j+1} ^{\infty}      k^{2\alpha -2}.
  \]
  Taking into account that   $ \int_{0 } ^{ 1\wedge(n \tau -j )} \gamma_{\frac {x+j}n}  dx  $ is bounded by $\frac {2} {2\alpha+1} (t_1 \vee t_2)^{2\alpha+1}$, we deduce that
  \[
     \|    N^{n,5} _\tau\|_2\le  C  \frac 1n    \sum_{j=0}^{\lfloor nT \rfloor}    \sum_{k=j+1} ^{\infty}      k^{2\alpha -2} \to 0,
    \]
    which converges to zero a $n\to \infty$.
  As a consequence,
  \begin{align*}
    \lim_{n\rightarrow \infty}    N^{n,4} _\tau&  =  \lim_{n\rightarrow \infty} 
   n^{-1}  \sum_{j=0} ^{\lfloor n\tau \rfloor}  (\sigma'\sigma)^2(X_{\frac j n})   \int_{0 } ^{ 1\wedge(n \tau -j )}  \gamma_{\frac {x+j}n}  \sum_{k=1} ^{\infty}   
  [ (x + k)^\alpha -  k^{\alpha} ] ^2 dx \\
  &= \left( \sum_{k=1} ^\infty  \int_0^1 [(x+k) ^\alpha -k^\alpha]^2 dx  \right)\int_0^{\tau \wedge t_1 \wedge t_2  } (t_1-s)^{\alpha}  (t_2-s)^{\alpha}(\sigma'\sigma)^2(X_s) ds,
  \end{align*}
where the convergence holds in $L^2$, uniformly in $\tau\in [0,T]$. This  implies the  desired result.

    \medskip \noindent
{\it Proof of  \eqref{limit3a}:}
Set
\begin{align*}
 L^n_\tau  &:=n^{2\alpha+ 1}\int^{\tau}_0\gamma_s\left(\sigma'(X_s)\right)^2\sigma(X_s)  
 \left( \int_0^s \psi_{n,1}(u,s) dW_u \right) \\
 & \qquad \times \left(\int^s_{\eta_n(s)}(s-\eta_n(u))^{\alpha}\sigma(X_{\eta_n(u)})\,dW_u\right)  ds \\
 &= n^{2\alpha+ 1}\int^{\tau}_0\gamma_s\left(\sigma'(X_s)\right)^2\sigma(X_s)  \sigma(X_{\eta_n(s)} ) (s-\eta_n(s))^\alpha
  (W_s-W_{\eta_n(s)}) \\
  & \qquad \times  \left( \int_0^s \psi_{n,1}(u,s) dW_u \right)  ds.
   \end{align*}
   Using the Lipschitz property of $\sigma$ and the H\"older continuity of $\sigma'$ and taking into account Lemma \ref{lema2}, we have that
   $ L^n_\tau$ has the same  asymptotic behavior in $L^2$, uniformly in $\tau \in [0,T]$,  as
\[
 L^{n,1}_\tau  
 = n^{2\alpha+ 1}\int^{\tau}_0\gamma_s (\sigma'\sigma)^2(X_{\eta_n(s)})  (s-\eta_n(s))^\alpha
  (W_s-W_{\eta_n(s)})   \left( \int_0^s \psi_{n,1}(u,s) dW_u \right)  ds.
\]
 Consider the decomposition $ L^{n,1}_\tau=:L^{n,2}_\tau + L^{n,3}_\tau+L^{n,4}_\tau $, where
 \[
 L^{n,2}_\tau  
 = n^{2\alpha+ 1}\int^{\tau}_0\gamma_s (\sigma'\sigma)^2(X_{\eta_n(s)})  (s-\eta_n(s))^\alpha
  (W_s-W_{\eta_n(s)})   \left( \int_0^{\eta_n(s)} \psi_{n,1}(u,s) dW_u \right)  ds,
\]
 \begin{align*}
 L^{n,3}_\tau  
 &= n^{2\alpha+ 1}\int^{\tau}_0\gamma_s (\sigma'\sigma)^2(X_{\eta_n(s)})  (s-\eta_n(s))^\alpha\\
 &  \qquad \times \left[ (W_s-W_{\eta_n(s)})   \left( \int_{\eta_n(s)}^s \psi_{n,1}(u,s) dW_u \right)  -\int_{\eta_n(s)}^s \psi_{n ,1}(u,s) du \right]  ds.
\end{align*}
and
\[
 L^{n,4}_\tau  
 = n^{2\alpha+ 1}\int^{\tau}_0\gamma_s (\sigma'\sigma)^2(X_{\eta_n(s)})  (s-\eta_n(s))^\alpha \int_{\eta_n(s)}^s \psi_{n ,1}(u,s) du ds.
\]
 We claim that
 \begin{equation} \label{p1}
 \lim_{n\rightarrow \infty}  \sup_{\tau\in [0,T]} \| L^{n,2}_\tau +L^{n,3}_\tau  \|_2=0.
 \end{equation}
 Indeed, we can express the sum of these two terms as follows
 \[
 L^{n,2}_\tau +L^{n,3}_\tau
 = n^{2\alpha+ 1}\int^{\tau}_0\gamma_s (\sigma'\sigma)^2(X_{\eta_n(s)})  (s-\eta_n(s))^\alpha  \Lambda^n_s ds,
 \]
 where
 \begin{align*}
 \Lambda^n_s& = (W_s-W_{\eta_n(s)})   \left( \int_0^s \psi_{n,1}(u,s) dW_u \right)  \\
 & \qquad  +(W_s-W_{\eta_n(s)})   \left( \int_{\eta_n(s)}^s \psi_{n,1}(u,s) dW_u \right)  -\int_{\eta_n(s)}^s \psi_{n ,1}(u,s) du.
 \end{align*}
Notice that $\Lambda^n_s $ is a random variable in the second Wiener chaos that satisfies $E[ \Lambda^n_s | \mathcal{F}_{\eta_n(s)}]=0$.
  As a consequence,  if $s_1+\frac 1n <s_2$, then $s_1 < \eta_n(s_2)$ we have
    \begin{align*}
  &  E[(\sigma'\sigma)^2(X_{\eta_n(s_1)})\Lambda^n_{s_1}(\sigma'\sigma)^2(X_{\eta_n(s_2)})\Lambda^n_{s_2} ]\\
  &=E[(\sigma'\sigma)^2(X_{\eta_n(s_1)})\Lambda^n_{s_1}(\sigma'\sigma)^2(X_{\eta_n(s_2)})E[\Lambda^n_{s_2} |\mathcal{F}_{\eta_n(s_2)}] ]=0.
\end{align*}
This implies that
   \begin{align*}
E[|  L^{n,2}_\tau +L^{n,3}_\tau|^2]
&= n^{4\alpha+ 2}
\int_{[0,\tau]^2} \gamma_{s_1} \gamma_{s_2}  \mathbf{1}_{ |s_1-s_2|\le\frac 1n}  (s_1-\eta_n(s_1))^\alpha(s_2-\eta_n(s_2))^\alpha \\
& \qquad \times 
E[(\sigma'\sigma)^2(X_{\eta_n(s_1)})\Lambda^n_{s_1}(\sigma'\sigma)^2(X_{\eta_n(s_2)})\Lambda^n_{s_2} ] ds_1ds_2
 \end{align*}
 We can write
\[
 \Lambda^n_s = (W_s-W_{\eta_n(s)})   \left( \int_0^s \psi_{n,1}(u,s) dW_u \right)   +
 I_2\left( \mathbf{1}_{[\eta_n(s),s]} \otimes \psi_{n,1}(\cdot,s) \right) ,
 \]
where $I_2$ denotes the double  Wiener-It\^o stochastic integral with respect to $W$. Therefore,
$\|  \Lambda^n_s \|_2 \le C n^{-\alpha -1}$, and we obtain the following estimate
\[
E[|  L^{n,2}_\tau +L^{n,3}_\tau|^2] \le  
\int_{[0,T]^2} \gamma_{s_1} \gamma_{s_2}  \mathbf{1}_{ \{|s_1-s_2|\le\frac 1n\}}   ds_1ds_2,
\]
 which converges to zero as $n$ tends to infinity. This proves \eqref{p1}.
 It remains to study the limit of the term $L^{n,4} _\tau$.
 We have
\begin{align*}
L^{n,4} _\tau &=
 n^{2\alpha+ 1} \sum_{j=0} ^{\lfloor n\tau \rfloor}   \int_{\frac jn}^{ \frac {j+1}n \wedge \tau}   \gamma_s (\sigma'\sigma)^2(X_{\frac jn})  (s-\frac jn)^\alpha \int_{\frac jn}^s  [(s- \frac jn)^\alpha - (s-u) ^\alpha]du ds\\
 &= \frac {\alpha } {\alpha+1} n^{2\alpha+ 1} \sum_{j=0} ^{\lfloor n\tau \rfloor}   \int_{\frac jn}^{ \frac {j+1}n \wedge \tau}   \gamma_s (\sigma'\sigma)^2(X_{\frac jn})  (s-\frac jn)^{2\alpha +1}  ds.
\end{align*}
We can replace $\gamma_s$ by $\gamma_{\frac jn}$, and we conclude that
\[
 \lim_{n\rightarrow \infty } L^{n,4} _\tau  = \frac {\alpha} {2(\alpha+1)^2}  \int_0^\tau \gamma_s (\sigma'\sigma)^2(X_s) ds,
 \]
 where the convergence holds in $L^2$, uniformly in $\tau\in [0,T]$.
 
     \medskip \noindent
{\it Proof of  \eqref{limit3b}:}
Set
\begin{align*}
 T^n_\tau  &:=n^{2\alpha+ 1}\int^{\tau}_0\gamma_s\left(\sigma'(X_s)\right)^2\sigma(X_s)  
 \left( \int_0^s \psi_{n,1}(u,s) dW_u \right) \\
 & \qquad \times \left(\int_0^{\eta_n(s)}\psi_{n,2}(u,s)\sigma(X_{\eta_n(u)})\,dW_u\right)  ds.
    \end{align*}
By the same arguments as before,  it suffices to analyze the asymptotic behavior of
\begin{align*}
 T^{n,1}_\tau  &:=n^{2\alpha+ 1}\int^{\tau}_0\gamma_s (\sigma'\sigma)^2(X_{\eta_n(s)})  
\int_0^{\eta_n(s)}\psi_{n,1}(u,s)  \psi_{n,2}(u,s) du ds\\
&=n^{2\alpha+ 1} \sum_{j=0} ^{\lfloor n\tau \rfloor}   \int_{\frac jn}^{ \frac {j+1}n \wedge \tau} 
\gamma_s (\sigma'\sigma)^2(X_{\frac jn})  \\
&\qquad \times 
\sum_{i=0}^{j-1}   \int_{\frac in} ^{\frac {i+1}n} [(s-\frac in)^\alpha- (s-u)^\alpha][ (s -\frac in)^\alpha -( \frac jn -\frac in)^\alpha] duds.
\end{align*}
Making the change of variables  $ns=x$,  $nu=v$  and later $x-j=y$ and $v-i=z$ allows us to write
\begin{align*}
 T^{n,1}_\tau  &=\frac 1n   \sum_{j=0} ^{\lfloor n\tau \rfloor}   \int_{ j}^{ (j+1) \wedge  n\tau} 
\gamma_{\frac xn} (\sigma'\sigma)^2(X_{\frac jn})  
\sum_{i=0}^{j-1}   \int_{i} ^{i+1} [(x-i)^\alpha- (x-v)^\alpha][ (x -i)^\alpha -( j-i)^\alpha] dvdx\\
&= \frac 1n  \sum_{j=0} ^{\lfloor n\tau \rfloor}   \int_{ 0}^{ 1 \wedge  (n\tau-j)} 
\gamma_{\frac {y+j}n} (\sigma'\sigma)^2(X_{\frac jn})   \\
& \qquad \times 
\sum_{i=0}^{j-1}   \int_{0} ^{1} [(y+j-i)^\alpha- (y+j-i-z)^\alpha][ (j -i+y)^\alpha -( j-i)^\alpha] dz dy \\
&=   \frac 1n \sum_{j=0} ^{\lfloor n\tau \rfloor}   \int_{ 0}^{ 1 \wedge  (n\tau-j)} 
\gamma_{\frac {y+j}n} (\sigma'\sigma)^2(X_{\frac jn})   \\
& \qquad \times 
\sum_{k=1}^{j}   \int_{0} ^{1} [(y+k)^\alpha- (y+k-z)^\alpha][ (k+y)^\alpha -k^\alpha] dz dy,
\end{align*}
which converges in $L^2$, uniformly in $\tau\in [0,T]$, as $n $ tends to infinity, to
\[
\left(\sum_{k=1}^{\infty}  \int_0^1  \int_{0} ^{1} [(y+k)^\alpha- (y+k-z)^\alpha][ (k+y)^\alpha -k^\alpha] dz dy\right)
\int_0^\tau \gamma_s (\sigma'\sigma)^2(X_s)ds.
\]

\medskip \noindent
{\it Proof of  \eqref{limit4}:}
Set
\begin{align*}
 G^n_\tau  &:=n^{2\alpha+ 1}\int^{\tau}_0\gamma_s\left(\sigma'(X_s)\right)^2\, \left(\int^s_{\eta_n(s)}(s-\eta_n(u))^{\alpha}\sigma(X_{\eta_n(u)})\,dW_u\right)\,\\&\qquad \times \left(\int^{\eta_n(s)}_0 \psi_{n,2}(u,s)  \sigma(X_{\eta_n(u)})\, dW_u\right) ds,
   \end{align*}
   where $\psi_{n,2}(u,s)$ is defined  in \eqref{psi2}. With the notation
   \[
   M^{n,1}_s=\int^s_{\eta_n(s)}(s-\eta_n(u))^{\alpha}\sigma(X_{\eta_n(u)})\,dW_u
   \]
   and
   \[
   M^{n,2}_s=\int^{\eta_n(s)}_0 \psi_{n,2} (u,s)  \sigma(X_{\eta_n(u)})\, dW_u,
   \]
   we can write
   \[
G^n_\tau=n^{2\alpha+ 1}\int^{\tau}_0 \gamma_s  \left(\sigma'(X_s)\right)^2 M^{n,1}_s M^{n,2}_sds.
   \]
   In order to show that this quantity tends to zero in $L^2$, uniformly in $\tau \in [0,T]$,  we will  show that for any interval $[\tau_1, \tau_2] \subset [0,\tau]$, we have
   \begin{equation} \label{dn2}
   \lim_{n\rightarrow \infty}n^{2\alpha+ 1}\int^{\tau_2}_{\tau_1} \gamma_s M^{n,1}_s M^{n,2}_s ds=0,
   \end{equation}
   where the convergence holds in $L^2$ uniformly in $\tau_1, \tau_2$ and $\tau$,    the case of the random factor  can be handled as before by the usual small blocks / big blocks argument. For the sake of simplicity we omit the details.
In order to show \eqref{dn2}, we notice that $E[M^{n,1}_{s_1} M^{n,2}_{s_1}M^{n,1}_{s_2} M^{n,2}_{s_2}]=0$ if $|s_1- s_2| >\frac 1n$. 
    In fact, if, for instance, $s_1+\frac 1n <s_2$, then $\eta_n(s_1) < \eta_n(s_2)$ and
    \[
    E[M^{n,1}_{s_1} M^{n,2}_{s_1}M^{n,1}_{s_2} M^{n,2}_{s_2}]=E[M^{n,1}_{s_1} M^{n,2}_{s_1}  M^{n,2}_{s_2}  E[M^{n,1}_{s_2} | \mathcal{F}_{\eta_n(s_2)}]]=0.
    \]
    As a consequence,
    \begin{align*} 
   &n^{4\alpha+ 2}  E\left[ \left| \int^{\tau_2}_{\tau_1} \gamma_s M^{n,1}_s M^{n,2}_s ds \right|^2 \right] \\
  & \qquad  =   n^{4\alpha+ 2} \int_{[\tau_1,\tau_2]^2} \gamma_{s_1} \gamma_{s_2}  \mathbf{1}_{\{ |s_1-s_2|\le\frac 1n\}}  E[M^{n,1}_{s_1} M^{n,2}_{s_1}M^{n,1}_{s_2} M^{n,2}_{s_2}] ds_1ds_2    \\
   & \qquad  \le   n^{4\alpha+ 2} \int_{[\tau_1,\tau_2]^2}  \gamma_{s_1} \gamma_{s_2}  \mathbf{1}_{ \{|s_1-s_2|\le\frac 1n\}}  \| M^{n,1}_{s_1}\|_4 \| M^{n,2}_{s_1}\|_4 \|M^{n,1}_{s_2} \|_4 \|M^{n,2}_{s_2}\|_4 ds_1ds_2.
   \end{align*}
    Taking into account that for all $s\in [0,\tau]$,   $ \| M^{n,1}_{s}\|_4 \le C n^{-\alpha -\frac 12}$ and  $ \| M^{n,2}_{s}\|_4 \le C 
   ( \left (\int_{0} ^{\eta_n(s)} \psi_{n,2}^2(u,s) du \right)^{\frac  12}$, we obtain, using Lemma \ref{lema1}
    \begin{align*}
   & n^{4\alpha+ 2}  E\left[ \left| \int^{\tau_2}_{\tau_1}   \gamma_{s}   M^{n,1}_s M^{n,2}_s ds \right|^2 \right]  \\
 & \qquad   \le C n^{2\alpha+1} 
    \int_{[\tau_1,\tau_2]^2}  \gamma_{s} \gamma_{s_2} \mathbf{1}_{\{ |s_1-s_2|\le\frac 1n\}}    \int_{0} ^{\eta_n(s)} \psi_{n,2}^2(u,s) du  \\
    & \qquad \le C  \int_{[0,T]^2}  \gamma_{s_1} \gamma_{s_2} \mathbf{1}_{\{ |s_1-s_2|\le\frac 1n\}}   ds_1 ds_2,
     \end{align*}
     which clearly converges to zero as $n\to \infty$. This completes the proof of  \eqref{limit4}.

    \medskip
\noindent
{\it Proof of {\bf (C4)}}:  
As in the proof of {\bf (C3)} we can replace $\Theta^n_s$ defined in  \eqref{theta1}    by $\widehat{\Theta}_s^n$ defined in \eqref{hattheta}.
  Then, condition (C4) will be a consequence of the following convergences in $L^2$, which are uniform in $\tau\in [0,t]$:
\begin{equation}  \label{lim10}
\lim_{n\rightarrow \infty}   n^{\alpha+\frac 12} \int_0^{\tau}  (t-s)^\alpha
\sigma'(X_s)  
 \left( \int_0^s \psi_{n,1}(u,s) dW_u \right) ds=0,
 \end{equation}
 \begin{equation}  \label{lim11}
\lim_{n\rightarrow \infty}   n^{\alpha+\frac 12}\int_0^{\tau}  (t-s)^\alpha
\sigma'(X_s)  
 \left( \int^s_{\eta_n(s)}\left(s-\eta_n(u)\right)^{\alpha}\sigma(X_{\eta_n(u)})\, dW_u \right) ds=0
 \end{equation}
 and
  \begin{equation}  \label{lim12}
\lim_{n\rightarrow \infty}   n^{\alpha+\frac 12} \int_0^{\tau}  (t-s)^\alpha
\sigma'(X_s)  
 \left( \int^{\eta_n(s)}_0 \psi_{n,2}(u,s)\, \sigma(X_{\eta_n(u)})\, dW_u \right) ds=0.
 \end{equation}
 
 \medskip \noindent
{\it Proof of  \eqref{lim10}:}  We will show first that for any interval $[\tau_1, \tau_2] \subset [0,\tau]$,
the term
\[
S^n=n^{\alpha+\frac 12} \int_{[\tau_1, \tau_2]}  (t-s)^\alpha
 \left( \int_0^s \psi_{n,1}(u,s) dW_u \right) ds
 \]
converges in $L^2$ to zero, uniformly in $\tau_1, \tau_2$ and $\tau$, as $n$ tends to infinity. We have
\begin{align*}
E[|S^n|^2]&= 2n^{2\alpha+1}\int_ {\tau_1} ^{\tau_2} \int_{\tau_1}^{s_2}  (t-s_1)^\alpha(t-s_2)^\alpha
 \int_0^{s_1}  \psi_{n,1}(u,s_1)  \psi_{n,1}(u,s_2)duds_1 ds_2\\
 &\le 2n^{2\alpha+1}\int_ {\tau_1} ^{\tau_2} \int_0^{s_2}  (t-s_1)^\alpha(t-s_2)^\alpha\\
 & \qquad \times
\sum_{i=0}^{ \lfloor ns_1 \rfloor}  \int_{\frac in} ^{\frac {i+1}n \wedge s_1}
[(s_1 -\frac in)^\alpha- (s_1 -u)^\alpha]  [(s_2 -\frac in)^\alpha- (s_2 -u)^\alpha]duds_1 ds_2.
\end{align*}
 With the change of variables $nu=x+i$, we obtain
 \begin{align*}
E[|S^n|^2]&\le 2 \int_ {\tau_1} ^{\tau_2} \int_0^{s_2}  (t-s_1)^\alpha(t-s_2)^\alpha
\sum_{i=0}^{ \lfloor ns_1 \rfloor}  \int_{0} ^{1\wedge (ns_1-i)}
[(ns_1 -i)^\alpha- (ns_1 -i- x)^\alpha] \\
& \qquad \times  [(ns_2 -i)^\alpha- (ns_2 -i-x)^\alpha]dxds_1 ds_2\\
&\le 2 \int_ {0} ^{t} \int_0^{s_2}  (t-s_1)^\alpha(t-s_2)^\alpha \\
& \qquad \times \sum_{k=0}^{ \infty}  \int_{0} ^{1\wedge (ns_1 -\lfloor ns_1 \rfloor+k)}
[(ns_1- \lfloor ns_1 \rfloor +k)^\alpha- (ns_1- \lfloor ns_1 \rfloor+k -x)^\alpha] \\
& \qquad \times  [(ns_2 -\lfloor ns_1 \rfloor+k)^\alpha- (ns_2-\lfloor ns_1 \rfloor+k -x)^\alpha]dxds_1 ds_2,
\end{align*}
which converges to zero as $n$ tends to infinity.
 
 Set
 \[
 S^{n,1}_\tau:=
   n^{\alpha+\frac 12} \int_0^{\tau}  (t-s)^\alpha
\sigma'(X_s)  
 \left( \int_0^s \psi_{n,1}(u,s) dW_u \right) ds.
 \]
 Fix an integer $M\ge 1$ and consider a partition $0=\tau_0 <\tau_1 < \cdots < \tau_M =\tau$, where
$ \tau_i= \frac iM \tau$ for $i=0,1, \dots, M$. We can write
\[
 S^{n,1}_\tau=  S^{n,M,1} _\tau  + S^{n,M,2} _\tau,
 \]
  where
  \[
  S^{n,M,1} _\tau  =
 n^{\alpha+\frac 12}  \sum_{i=0} ^{M-1} \sigma'(X_{\tau_i})  \int_{\tau_i} ^{\tau_{i+1}}     (t-s)^\alpha
 \left( \int_0^s \psi_{n,1}(u,s) dW_u \right) ds
\]
and
 \[
  S^{n,M,2} _\tau  =
 n^{\alpha+\frac 12}  \sum_{i=0} ^{M-1}   \int_{\tau_i} ^{\tau_{i+1}}     (t-s)^\alpha [\sigma'(X_s)- \sigma'(X_{\tau_i})]
 \left( \int_0^s \psi_{n,1}(u,s) dW_u \right) ds
\]
From Step 1, we know that for any fixed $M$
\begin{equation}  \label{L31}
\lim_{n\rightarrow \infty}  \sup_{\tau \in [0,t]}E[|  S^{n,M,1} _\tau|^2]=0.
\end{equation}
Assuming that $ n\ge M$,  the term    $S^{n,M,2} _\tau $ can be estimated as follows
\[
 \| S^{n,M,2} _\tau  \|_2 \le C \sup_{u,v \in [0,t], |u-v| \le \frac {2t}M} \|  (\sigma')^2(X_u) - (\sigma')^2(X_v)  \|_4
n^{\alpha+\frac 12}  \int_0^\tau   (t-s)^\alpha   \left \|  \int_0^s \psi_{n,1}(u,s) dW_u  \right \|_4 ds
\]
From Lemma  \ref{lema2} we get
\[
\sup_{n\ge 1} n^{\alpha+\frac 12}  \int_0^t  (t-s)^\alpha    \| \int_0^s \psi_{n,1}(u,s) dW_u \|_4 ds <\infty.
\]
Therefore,
\begin{equation} \label{L41}
\lim_{M\rightarrow \infty} \sup_{n\ge M} \sup_{\tau\in [0,t]} \| S^{n,M,2} _\tau  \|_2=0.
\end{equation}
In conclusion, \eqref{L31} and \eqref{L41} imply  \eqref{lim10}.

\medskip \noindent
{\it Proof of  \eqref{lim11}:}
We have
\begin{align*}
   P^n_\tau&: =n^{\alpha+\frac12}\int^{\tau}_0(t-s)^{\alpha} \sigma'(X_s)\, 
   \left(\int^s_{\eta_n(s)}(s-\eta_n(u))^{\alpha}\sigma(X_{\eta_n(u)})\,dW_u\right)ds\\
   &=  n^{\alpha+\frac12}\int^{\tau}_0(t-s)^{\alpha} \sigma'(X_{s})  \sigma(X_{\eta_n(s)})  \Xi^{n,2}_s  ds,
   \end{align*}
   where $\Xi^{n,2}_s=\int^s_{\eta_n(s)}(s-\eta_n(u))^{\alpha} dW_u$. Consider the decomposition
   \[
P^n_\tau = P^{n,1}_\tau + P^{n,2}_\tau,
      \]
      where
      \[
      P^{n,1}_\tau =n^{\alpha+\frac12}\int^{\tau}_0(t-s)^{\alpha}  (\sigma'\sigma)(X_{\eta_n(s)})  \Xi^{n,2}_s  ds
      \]
      and
        \[
      P^{n,2}_\tau =n^{\alpha+\frac12}\int^{\tau}_0(t-s)^{\alpha}  [(\sigma'(X_s)- \sigma'(X_{\eta_n(s)})]\sigma(X_{\eta_n(s)})  \Xi^{n,2}_s  ds.
      \]
   The $\beta$-H\"older continuity of $\sigma'$,  Lemma \ref{rglo} and the fact that  $ \|\Xi^{n,2}_s   \|_2 \le C n^{-\alpha -\frac 12}$ imply that
   \[
   \lim_{n\rightarrow \infty} \sup_{\tau \in [0,t]} E[|      P^{n,2}_\tau|^2]=0.
   \]
   To handle the term  $P^{n,1}_\tau$ we write
   \[
   E[|      P^{n,1}_\tau|^2]=n^{2\alpha+1}\int_{[ 0,\tau]^2} (t-s_1)^{\alpha} (t-s_2)^{\alpha}   E\left[(\sigma'\sigma)(X_{\eta_n(s_1)}) 
   (\sigma'\sigma)(X_{\eta_n(s_2)})   \Xi^{n,2}_{s_1} \Xi^{n,2}_{s_2}  \right]   ds_1 ds_2
   \]
We claim that the expectation inside the integral vanishes if $|s_1-s_2 | >\frac 1n$. Indeed, if, for instance $s_1+\frac 1n <s_2$, then $\eta_n(s_1) < \eta_n(s_2)$ and
\begin{align*}
&E\left[(\sigma'\sigma)(X_{\eta_n(s_1)}) 
   (\sigma'\sigma)(X_{\eta_n(s_2)})   \Xi^{n,2}_{s_1} \Xi^{n,2}_{s_2}  \right]  \\
   & \qquad =
   E\left[(\sigma'\sigma)(X_{\eta_n(s_1)}) 
   (\sigma'\sigma)(X_{\eta_n(s_2)})   \Xi^{n,2}_{s_1}  E[ \Xi^{n,2}_{s_2} | \mathcal{F}_{\eta_n(s_2)} \right]]=0.
   \end{align*}
   Therefore,
      \begin{align*}
    E[|      P^{n,1}_\tau|^2]  & =n^{2\alpha+1}\int_{[ 0,\tau]^2} (t-s_1)^{\alpha} (t-s_2)^{\alpha}  \mathbf{1}_{\{|s_1-s_2| \le \frac 1n\}}  \\
    & \qquad \times E\left[(\sigma'\sigma)(X_{\eta_n(s_1)}) 
   (\sigma'\sigma)(X_{\eta_n(s_2)})   \Xi^{n,2}_{s_1} \Xi^{n,2}_{s_2}  \right]   ds_1 ds_2 \\
   &  \le  C n^{2\alpha+1}\int_{[ 0,\tau]^2} (t-s_1)^{\alpha} (t-s_2)^{\alpha}  \mathbf{1}_{\{|s_1-s_2| \le \frac 1n\}}   \|   \Xi^{n,2}_{s_1} \|_4 \| \Xi^{n,2}_{s_2} \|_4    ds_1 ds_2 \\
   &  \le C \int_{[ 0,t]} (t-s_1)^{\alpha} (t-s_2)^{\alpha}  \mathbf{1}_{\{|s_1-s_2| \le \frac 1n\}}     ds_1 ds_2,
   \end{align*}
   which tends to zero as $n\to \infty$. This completes the proof of \eqref{lim11}.
      
\medskip \noindent
{\it Proof of  \eqref{lim12}:}
Set
\[
 F^n_\tau =n^{\alpha+ \frac12}\int^{\tau}_0(t-s)^{\alpha} \sigma'(X_s)\, \left(\int^{\eta_n(s)}_0 \psi_{n,2}(u,s) \sigma(X_{\eta_n(u)})\, dW_u\right) ds,
   \]
   where $\psi_{n,2}(u,s)$ is defined in \eqref{psi2}.  Fix $\delta>0$ such that $\delta <t$ and suppose that $n$ satisfies $\delta+\frac 1n <t$ and  $\frac 1n <\delta$. We decompose the term
   $  F^n_\tau$ as follows
   \begin{align*}
F^n_\tau  &=   n^{\alpha+ \frac12}\int^{(\delta+\frac 1n)\wedge \tau}_0(t-s)^{\alpha} \sigma'(X_{s}) \left(\int^{\eta_n(s)}_0 \psi_{n,2}(u,s) \sigma(X_{\eta_n(u)})\, dW_u\right) ds  \\
 & \qquad +   
 n^{\alpha+ \frac12}\int^{\tau}_{(\delta+\frac 1n)\wedge \tau}(t-s)^{\alpha} [\sigma'(X_s)- \sigma'(X_{\eta_n(s)-\delta})]\, \left(\int^{\eta_n(s)}_0 \psi_{n,2}(u,s) \sigma(X_{\eta_n(u)})\, dW_u\right) ds \\
    & \qquad + n^{\alpha+ \frac12}\int_{(\delta+\frac 1n)\wedge \tau} ^\tau(t-s)^{\alpha} \sigma'(X_{\eta_n(s)-\delta}) \left(\int^{\eta_n(s)-\delta}_0 \psi_{n,2}(u,s) \sigma(X_{\eta_n(u)})\, dW_u\right) ds  \\
    & \qquad + n^{\alpha+ \frac12}\int_{(\delta+\frac 1n)\wedge \tau} ^\tau(t-s)^{\alpha} \sigma'(X_{\eta_n(s)-\delta})\left(\int_{\eta_n(s)-\delta}^{\eta_n(s)} \psi_{n,2}(u,s) \sigma(X_{\eta_n(u)})\, dW_u\right) ds\\
    &=:  B^{n,1} _\tau + B^{n,2} _\tau + B^{n,3} _\tau + B^{n,4} _\tau.
   \end{align*}
   We have the following estimates. Using Minkowkski's, H\"older and Burkholder's inequalities together with Lemma \ref{lema1}, we obtain
   \begin{align}  \nonumber
   \|B^{n,1} _\tau \|_2 & \le  C n^{\alpha+ \frac12}\int^{\delta+\frac 1n }_0(t-s)^{\alpha}   \left(\int^{\eta_n(s)}_0 \psi_{n,2}^2(u,s)   du\right) ^{\frac 12}ds \\ \nonumber
   & \le C \left( t^{\alpha+1} - (t- \delta -\frac 1n)^{\alpha+1} \right)\\  \label{b1}
   & \le C \left(\delta + \frac 1n \right) ^{(\alpha+1) \wedge 1}.
   \end{align}
   For the term  $B^{n,2} _\tau$ the $\beta$-H\"older continuity of $\sigma'$ allows us to write
   \begin{align} \nonumber
   \|B^{n,2} _\tau \|_2 & \le  C \sup_{|u-v| \le \delta +\frac 1n} \| X_u-X_v\|_4^{\beta} n^{\alpha+ \frac12}\int^{t}_0(t-s)^{\alpha}   \left(\int^{\eta_n(s)}_0 \psi_{n,2}^2(u,s)   du\right) ^{\frac 12}ds \\  \label{b2}
   & \le C   \left( \delta +\frac 1n  \right)^{\beta(\alpha+\frac 12)}.
   \end{align}
   For the term   $B^{n,3} _\tau$, we have
   \begin{equation}  \label{dn3}
      \|B^{n,3} _\tau \|_2  \le  C n^{\alpha+ \frac12}\int_{(\delta+\frac 1n)\wedge t} ^t(t-s)^{\alpha}   \left(\int^{\eta_n(s)-\delta}_0 \psi_{n,2}^2(u,s)  du \right)  ^{\frac 12} ds.  
      \end{equation}
      Applying the Mean Value Theorem and taking into account that $\delta >\frac 1n$, we obtain
      \begin{align} \nonumber
      \int^{\eta_n(s)-\delta}_0 \psi_{n,2}^2(u,s)  du  & =  \int^{\eta_n(s)-\delta}_0  [ (s- \eta_n(u))^{\alpha} - (\eta_n(s)- \eta_n(u) )^\alpha] ^2du  \\  \label{dn4}
      & \le C  n^{-2} \delta^{2\alpha-2}.
      \end{align}
      Substituting \eqref{dn4} into \eqref{dn3}, yields
      \begin{equation}  \label{b3}
         \|B^{n,3} _\tau \|_2  \le  n^{\alpha -\frac 12} \delta^{\alpha -1}.
         \end{equation}
         The estimation of the term $B^{n,4} _\tau$ is more involved. First, we write
         \begin{align*}
         E[|B^{n,4} _\tau|^2] =
          n^{2\alpha+ 1}\int_{[(\delta+\frac 1n)\wedge \tau,\tau]^2}
          (t-s_1)^{\alpha}  (t-s_2)^{\alpha} E\left[  \sigma'(X_{\eta_n(s_1)-\delta})  M_{s_1}^{n,3}   \sigma'(X_{\eta_n(s_2)-\delta})  M_{s_2}^{n,3}  \right]
          ds_1ds_2,
         \end{align*}
         where, for any $s\in [0,\tau]$, we use the notation
         \[
         M^{n,3}_s=\int_{\eta_n(s)-\delta}^{\eta_n(s)} \psi_{n,2}(u,s) \sigma(X_{\eta_n(u)})\, dW_u.
         \]
         We claim that for  $|s_1 -s_2 |> \delta+\frac 1n$, the expectation $
         E\left[  \sigma'(X_{\eta_n(s_1)-\delta})  M_{s_1}^{n,3}   \sigma'(X_{\eta_n(s_2)-\delta})  M_{s_2}^{n,3}  \right]$ vanishes. Indeed, if, for instance
         $s_1 < s_2-\delta -\frac 1n$, we have
         \begin{align*}
         &E\left[  \sigma'(X_{\eta_n(s_1)-\delta})  M_{s_1}^{n,3}   \sigma'(X_{\eta_n(s_2)-\delta})  M_{s_2}^{n,3}  \right] \\
        & \qquad  =E\left[  \sigma'(X_{\eta_n(s_1)-\delta})  M_{s_1}^{n,3}   \sigma'(X_{\eta_n(s_2)-\delta})  E[M_{s_2}^{n,3} | \mathcal{F}_{\eta_n(s_2)-\delta}] \right] =0.
         \end{align*}
         As a consequence,
         \begin{align}  \nonumber
          E[|B^{n,4} _\tau|^2] &=
          n^{2\alpha+ 1}\int_{[(\delta+\frac 1n)\wedge \tau,\tau]^2}
          (t-s_1)^{\alpha}  (t-s_2)^{\alpha}  \mathbf{1}_{\{|s_1-s_2| \le \delta +\frac 1n\}}\\ \nonumber
          & \qquad \times E\left[  \sigma'(X_{\eta_n(s_1)-\delta})  M_{s_1}^{n,3}   \sigma'(X_{\eta_n(s_2)-\delta})  M_{s_2}^{n,3}  \right]
          ds_1ds_2 \\ \nonumber
          & \le C  n^{2\alpha+ 1}\int_{[(\delta+\frac 1n)\wedge \tau,\tau]^2}
          (t-s_1)^{\alpha}  (t-s_2)^{\alpha}  \mathbf{1}_{\{|s_1-s_2| \le \delta +\frac 1n\}} \|M_{s_1}^{n,3}\|_4 \| M_{s_2}^{n,3} \|_4        ds_1ds_2 \\  \label{b4}
          & \le  C \int_{[0,t]^2} (t-s_1)^{\alpha}  (t-s_2)^{\alpha}  \mathbf{1}_{\{|s_1-s_2| \le \delta +\frac 1n\}} 
ds_1ds_2,
         \end{align}
         where in the last inequality we used Burkholder's inequality and Lemma \ref{lema1}.
         From the estimates \eqref{b1}, \eqref{b2}, \eqref{b3} and \eqref{b4}, letting first $n$ tend to infinity and later $\delta $ tend to zero, we obtain that
         \eqref{lim12} holds true.
   \end{proof}

   \subsection{Asymptotic behavior of $\wY^{n,1}_t$}
   We recall that   $\wY^{n,1}_t$ satisfies the equation
\[
\wY^{n,1}_t = \wA^{n,1}_t  + \widetilde{C}^n_t +\int_0^t(t-s)^{\alpha}\sigma'(X_s) \wY^{n,1}_s\,dW_s.
\]
In this section we will prove the following result.
\begin{proposition} \label{prop45}
The finite-dimensional distributions of the process $\{\wY^{n,1}_t, t\in [0,T]\}  $ converges in distribution to those of the process $\wY^{\infty,1} =\{\wY^{\infty,1}_t , t\in [0,T]\}$, 
where $\wY^{\infty,1} $ is the solution to the stochastic Volterra equation 
  \[
\wY^{\infty,1}_t = \kappa_2  \int^{t}_0(t-s)^{\alpha}(\sigma'\sigma)(X_s)dB_s   +\int_0^t(t-s)^{\alpha}\sigma'(X_s) \wY^{\infty,1}_{s}\,dW_s,
\]
where   $B$ is a standard Brownian motion independent of $W$.
\end{proposition}

\begin{proof}
Fix an integer $m\ge 1$. Set
\[
\wY^{n,m,1}_t = \wA^{n,1}_t  + \widetilde{C}^n_t +\int_0^t(t-s)^{\alpha}\sigma'(X_s) \wY^{n,m,1}_{\beta_m(s)}\,dW_s,
\]
 where $\beta_m(s) = \frac im $ if $\frac im  \le s < \frac {i+1}m $.  From Section  \ref{sec3} we know that, for any fixed $m\ge 1$,  the finite-dimensional distributions of the process 
 $\wY^{n,m,1}$ converges in distribution to those of the process   $\wY^{\infty,m,1} = \{\wY^{\infty,m,1} _t, t\in [0,T]\}  $ that satisfies
 \[
\wY^{\infty,m,1}_t = \kappa_2  \int^{t}_0(t-s)^{\alpha}(\sigma'\sigma)(X_s)dB_s   +\int_0^t(t-s)^{\alpha}\sigma'(X_s) \wY^{\infty,m,1}_{\beta_m(s)}\,dW_s,
\]
 where  $B$ is a standard Brownian motion independent of $W$.
Then, the proposition  follows immediately  form the following convergences:
 \begin{equation} \label{conv1}
 \lim_{m\rightarrow\infty} \sup_{n \ge 1}E[ |\wY^{n,1}_t -\wY^{n,m,1}_t |^2]=0
 \end{equation}
 and
 \begin{equation} \label{conv2}
 \lim_{m\rightarrow\infty}  E[ |\wY^{\infty,1}_t -\wY^{\infty,m,1}_t |^2]=0.
 \end{equation}
Let us first show  \eqref{conv1}.  We have
 \begin{align}  \nonumber
 E[ |\wY^{n,1}_t -\wY^{n,m,1}_t |^2] &= 
 \int_0^t(t-s)^{2\alpha} E \left[  (\sigma'(X_s))^2  [\wY^{n,1}_s-\wY^{n,m,1}_{\beta_m(s)}]^2\, \right]ds \\ \nonumber
 & \le \| \sigma'\|_\infty^2  \int_0^t(t-s)^{2\alpha}    E[ |\wY^{n,1}_s-\wY^{n,m,1}_{\beta_m(s)}|]^2\, ds \\ \nonumber
  & \le  2\| \sigma'\|_\infty^2  \int_0^t(t-s)^{2\alpha}    E[ |\wY^{n,1}_s-\wY^{n,m,1}_{s}|]^2\, ds  \\   \label{g1}
 & \qquad  +  2\| \sigma'\|_\infty^2  \int_0^t(t-s)^{2\alpha}    E[ |\wY^{n,m,1}_s-\wY^{n,m,1}_{\beta_m(s)}|]^2\, ds. 
 \end{align} 
  Applying Lemma \ref{Hol} and the fact that $\sup_{t\in [0,T]}\sup_{n,m\ge 1} E[|\wY^{n,m,1}_t|^2] <\infty$, we can show that the process   $\wY^{n,m,1}_t$ is H\"older continuous, uniformly in $n$ and $m$. Indeed,
  for any $0 \le t_1 <t_2 \le T$,
  \begin{align} \nonumber
 E[| \wY^{n,m,1}_{t_1}-\wY^{n,m,1}_{t_2}|^2] & \le  C E[|\wA^{n,1}_{t_1} + \widetilde{C}^n_{t_1}- \wA^{n,1}_{t_2} - \widetilde{C}^n_{t_2}|^2] \\ \nonumber
 & \qquad +C \int_0^{t_1}  |(t_2-s)^{ \alpha} -(t_1-s)^\alpha|^2   E[ |\wY^{n,m,2}_{\beta_m(s)}|^2] ds \\ \nonumber
 &  \qquad + C \int_{t_1}^{t_2}  (t_2-s)^{2\alpha}   E[ |\wY^{n,m,1}_{\beta_m(s)}|^2] ds\\ \label{g2}
 & \le C|t_1-t_2|^{2\alpha+1}.
 \end{align}  
Substituting \eqref{g2} into \eqref{g1}, we obtain
\[
 E[ |\wY^{n,1}_t -\wY^{n,m,1}_t |^2] 
   \le  2\| \sigma'\|_\infty^2  \int_0^t(t-s)^{2\alpha}    E[ |\wY^{n,1}_s-\wY^{n,m,1}_{s}|]^2\, ds
   +  C m^{-2\alpha-1}. 
\]
which implies \eqref{conv1}. The proof of the convergence \eqref{conv2} is analogous.
\end{proof}

\subsection{Conclusion}
Part (i) of Theorem  \ref{mainthm}  follows from the convergence \eqref{A1} and Proposition \ref{first_term}. Then, 
 the convergence \eqref{A1} and Proposition \ref{propY} imply that
 \[
  \lim_{n\rightarrow\infty} \sup_{t\in [0,T]}E[ |\wY^{n,1}_t Y^{n,1}_t |^2]=0.
  \]
  Therefore, Part (ii) of Theorem  \ref{mainthm}   is a consequence  of  Proposition \ref{prop45}.
  This concludes the proof of Theorem \ref{mainthm}.

\section{Proof of Theorem \ref{thmjoint}}

Theorem \ref{thmjoint} follows from the next proposition which ensures the asymptotic independence of the random variables $N^{n,1} _{\eta_n(t)}$ and $Y^{n,1}_t$.

\begin{proposition}
For any $t\in (0,T]$ and real numbers $\lambda, \mu \in \R$, the following convergence holds true
\[
\lim_{n\rightarrow \infty}
E \left( \exp  \left(i\lambda  N^{n,1}_{\eta_n(t)} + i \mu Y^{n,1}_{t} \right) \right)
 =E \left( \exp   \left(i \lambda Z_{t} \right) \right)   E\left( \exp  \left( i\mu Y^{\infty,1}_{t} \right)\right),
\]
where $N^{n,1} _t$ has been defined in \eqref{ene} and $Y^{n,1}_t$ is defined in \eqref{decom}.
\end{proposition}

\begin{proof}
Set
\[
 \Lambda_n:=\exp  \left(i\lambda  N^{n,1}_{\eta_n(t)} + i \mu Y^{n,1}_{t} \right).
 \]
 Recall that, for any $t\in [0,T]$,
  \[ 
  N^{n,1}_t= n^{\alpha+\frac 12}\int_0^{\eta_n(t)} \psi_{n,1}(s,\eta_n(t)) dW_s
 \]
 and
 \[
 Y^{n,1}_t=n^{\alpha+\frac 12} \int_0^{t} (t-s)^\alpha  \xi^n_sdW_s,
 \]
  where    $\xi^n_s=\sigma(X^n_{\eta_n(s)}) -\sigma(X_s)$.
Let $\delta>0$ be such that $\delta< t$. Assume that $\frac 1n <\frac  \delta 2$. We can decompose the random variable $\Lambda_n$ as follows 
\[
\Lambda_n= \Lambda^{(1)}_{n,\delta}+\Lambda^{(2)}_{n,\delta},
\]
where
\begin{align*}
\Lambda^{(1)}_{n,\delta} &=\exp  \left( i  \mu  Y^{n,1}_{t} +i  \lambda n^{\alpha+\frac 12}\int_{t- \delta} ^{\eta_n(t)}\psi_{n,1}(s,\eta_n(t)) dW_s
  \right) \\
  & \qquad \times \left[\exp  \left( i \lambda n^{\alpha+\frac 12}\int^{t- \delta} _0\psi_{n,1}(s,\eta_n(t)) dW_s \right) -1\right]
\end{align*}
and
\[
\Lambda^{(2)}_{n,\delta}=
\exp  \left( i\mu  Y^{n,1}_{t} +i  \lambda n^{\alpha+\frac 12} \int_{t- \delta} ^{\eta_n(t)}\psi_{n,1}(s,\eta_n(t)) dW_s
  \right).
\]
 We claim that for any fixed $\delta>0$, 
 \begin{equation} \label{L1}
 \lim_{n\rightarrow \infty} E[\Lambda^{(1)}_{n,\delta}]=0.
 \end{equation}
 Indeed, we can write
 \begin{align*}
 |E[\Lambda^{(1)}_{n,\delta}]| & \le   | \lambda| n^{\alpha+\frac 12} E \left[ \left| 
 \int^{t- \delta} _0\psi_{n,1}(s,\eta_n(t)) dW_s \right| \right] \\
 &\le  | \lambda| n^{\alpha+\frac 12}   \left( \int^{t- \delta} _0\psi^2_{n,1}(s,\eta_n(t)) ds \right)^{\frac 12}.
 \end{align*}
 By the Mean Value Theorem, taking into account that $s\le t-\delta$ and $\frac 1n < \frac \delta 2$, we obtain
 \begin{align*}
 \psi^2_{n,1}(s,\eta_n(t))& =[(\eta_n(t)- \eta_n(s))^\alpha- (\eta_n(t)-s)^\alpha]^2 \\
 & \le (\eta_n(t)-s)^{2\alpha-2} (s- \eta_n(s))^2
 \le( \delta/2) ^{2\alpha-2} n^{ -2},
 \end{align*}
   and the convergence \eqref{L1} follows easily. We now make a further decomposition
   \[
   \Lambda^{(2)}_n= \Lambda^{(3)}_{n,\delta}+\Lambda^{(4)}_{n,\delta},
\]
 where
 \begin{align*}
\Lambda^{(3)}_{n,\delta} &=
\exp  \left( i\mu   n^{\alpha+\frac 12} \int_0^{t-\delta} (t-s)^\alpha \xi^n_s dW_s +i  \lambda n^{\alpha+\frac 12} \int_{t- \delta} ^{\eta_n(t)}\psi_{n,1}(s,\eta_n(t)) dW_s
  \right) \\
  & \qquad \times \left[ \exp \left(i\mu   n^{\alpha+\frac 12} \int_{t-\delta}^t (t-s)^\alpha \xi^n_s dW_s  \right)-1 \right]
\end{align*}
and
\[
\Lambda^{(4)}_{n,\delta} =
\exp  \left( i\mu   n^{\alpha+\frac 12} \int_0^{t-\delta} (t-s)^\alpha \xi^n_s dW_s +i  \lambda n^{\alpha+\frac 12} \int_{t- \delta} ^{\eta_n(t)}\psi_{n,1}(s,\eta_n(t)) dW_s
  \right).
  \]
  We claim that  
 \begin{equation} \label{L2a}
 \lim_{\delta \rightarrow 0}\limsup_{n\rightarrow \infty} E[\Lambda^{(3)}_{n,\delta}]=0.
 \end{equation}
 Indeed, we can write
 \begin{align*}
 |E[\Lambda^{(3)}_{n,\delta}]| & \le   | \mu| n^{\alpha+\frac 12} E \left[ \left| 
 \int_{t- \delta} ^t (t-s)^\alpha \xi^n_s dW_s \right| \right] \\
 &\le  | \mu | n^{\alpha+\frac 12}   \left( \int_{t- \delta} ^t (t-s)^{2\alpha} E[ | \xi^n_s|^2] ds \right)^{\frac 12}.
 \end{align*}
 Taking into account that  $E[ | \xi^n_s|^2] ds \le C n^{-2\alpha-1}$, the limit  \eqref{L2a} follows immediately.
 Finally, we can write
 \begin{align*}
 & E[\Lambda^{(4)}_{n,\delta} ] = E[ E[\Lambda^{(4)}_{n,\delta} | \mathcal{F}_{t-\delta}]]  \\
 &=  E\left[ \exp  \left( i\mu   n^{\alpha+\frac 12} \int_0^{t-\delta} (t-s)^\alpha \xi^n_s dW_s \right) 
 E\left[  \exp \left( i  \lambda n^{\alpha+\frac 12} \int_{t- \delta} ^{\eta_n(t)}\psi_{n,1}(s,\eta_n(t)) dW_s \right) | \mathcal{F}_{t-\delta}\right] \right] \\
 &= E\left[ \exp  \left( i\mu   n^{\alpha+\frac 12} \int_0^{t-\delta} (t-s)^\alpha \xi^n_s dW_s \right)  \right]
 \exp \left(-\frac {\lambda^2} 2 n^{2\alpha+1} \int_{t- \delta} ^{\eta_n(t)}\psi_{n,1}^2(s,\eta_n(t)) ds \right).
 \end{align*}
 We claim that
 \begin{equation} \label{L3}
 \lim_{n\rightarrow \infty}\exp \left(-\frac {\lambda^2} 2 n^{2\alpha+1} \int_{t- \delta} ^{\eta_n(t)}\psi_{n,1}^2(s,\eta_n(t)) ds \right)
 = \exp\left( -\frac {\lambda^2} 2 \kappa_1^2 \right)
 \end{equation}
 and
 \begin{equation} \label{L4}
 \lim_{\delta \rightarrow 0} \limsup_{n\rightarrow \infty} \left|E\left[ \exp  \left( i\mu   n^{\alpha+\frac 12} \int_0^{t-\delta} (t-s)^\alpha \xi^n_s dW_s \right)  \right] -E \left[ \exp \left(i \mu Y^{n,1}_t \right) \right]  \right| =0.
 \end{equation}
 The convergence  \eqref{L3} follows from Proposition \ref{first_term} and the previous computations.  The proof of \eqref{L4} is analogous to that of \eqref{L2}. In conclusion, from  \eqref{L1},  \eqref{L2}, \eqref{L3} and  \eqref{L4},  we can write
 \begin{align*}
 \lim_{n\rightarrow \infty} E[ \Lambda_n] &=\exp\left( -\frac {\lambda^2} 2 \kappa_1 \right) \lim_{n\rightarrow \infty} E \left[ \exp \left(i \mu Y^{n,1}_t \right) \right] \\
 &=E\left[  \exp \left( i \lambda Z_t \right) \right]E \left[ \exp \left(i \mu Y^{\infty,1}_t \right) \right],
 \end{align*}
 which completes the proof of the proposition.
  \end{proof}

 \section{Appendix}
In the next lemma we establish the H\"older continuity in $L^2$ of the processes $\wA^{n,1}_t$ and $\wC_t$ defined in
\eqref{dn3} and \eqref{C2}, respectively.

\begin{lemma}  \label{Hol}
Consider the processes $\wA^{n,1}_t$ and $\wC_t$ defined in
\eqref{dn3} and \eqref{C2}, respectively.
There exists a constant $C>0$, such that for all $t_1, t_2\in [0,T]$,
\begin{equation} \label{h1}
E[| \wA^{n,1}_{t_1} -\wA^{n,1}_{t_2}|^2] \le C |t_1-t_2|^{2\alpha+1},
\end{equation}
and
\begin{equation} \label{h2}
E[| \wC_{t_1} -\wC_{t_2}|^2] \le C |t_1-t_2|^{2\alpha+1}.
\end{equation}
\end{lemma}

\begin{proof}
Let us first show \eqref{h1}. For any $0\le t_1<t_2\le T$ we have
\begin{align*}
E[| \wA^{n,1}_{t_1} -\wA^{n,1}_{t_2}|^2]  &=
n^{2\alpha+1} \int_0^{t_1}  |(t_2-s)^{ \alpha} -(t_1-s)^\alpha|^2 E\left[ (\sigma'\sigma)^2(X_s)    \left( \int_0^s \psi_{n,1}(u,s)dW_u \right)^2 \right] ds \\
&\qquad + n^{2\alpha+1} \int_{t_1} ^{t_2}  (t_2-s)^{2 \alpha} E\left[ (\sigma'\sigma)^2(X_s)    \left( \int_0^s \psi_{n,1}(u,s)dW_u \right)^2 \right] ds  \\
&\le  C n^{2\alpha+1} \int_0^{t_1}  |(t_2-s)^{ \alpha} -(t_1-s)^\alpha|^2    \int_0^s \psi^2_{n,1}(u,s)  du  ds \\
&\qquad + C n^{2\alpha+1} \int_{t_1} ^{t_2}  (t_2-s)^{2 \alpha}   \int_0^s \psi^2_{n,1}(u,s) du  ds .
\end{align*}
Applying Lemma \ref{lema2}, yields,
\begin{align*}
E[| \wA^{n,1}_{t_1} -\wA^{n,1}_{t_2}|^2]  &\le C  \int_0^{t_1}  |(t_2-s)^{ \alpha} -(t_1-s)^\alpha|^2   ds + C  \int_{t_1} ^{t_2}  (t_2-s)^{2 \alpha}ds  \\
&\le  C |t_1-t_2|^{2\alpha+1}.
\end{align*}
The second inequality in the previous display can be deduced from the $L^2$ norm of the increment on the interval $[t_1, t_2]$ of a fractional Brownian motion with Hurst parameter $H=\alpha +\frac 12$. The proof of the inequality \eqref{h2} is analogous, but instead of  Lemma \ref{lema2} we make use of inequality \eqref{rgla}. \end{proof}

The next results are technical lemmas that have been used several time in the proofs.

\begin{lemma} \label{alpha1}
For any $\alpha \in (-\frac{1}{2},\frac12) $ and for any $\delta >0$,there are  positive constants $C_1$, $C_2$ and $C_3$, such that
for all $t\in [0,1]$, 
\begin{equation} \label{ineq1}
\int^{t}_0\left[(t-\eta_n(s))^{\alpha}-(t-s)^{\alpha}\right]^2\, ds
\leq C_1\, n^{-2\alpha-1}
\end{equation}
and
\begin{equation} \label{ineq2}
\int^{(t-\delta)_+}_0\left[(t-\eta_n(s))^{\alpha}-(t-s)^{\alpha}\right]^2\, ds
\leq C_2 n^{-2} \delta^{2\alpha-1}.
\end{equation}
\end{lemma}

\begin{proof}
By the Mean Value Theorem, for any $0 \le s<t \le 1$, we have
\begin{equation} \label{e1}
\left[(t-\eta_n(s))^{\alpha}-(t-s)^{\alpha}\right]^2\leq (t-s)^{2\alpha-2}(s-\eta_n(s))^2.
\end{equation}
Let us first show the inequality \eqref{ineq1}. We can assume that $t>\frac 1n$, because in the case $t\le \frac 1n$, the inequality can be easily proved. We decompose the interval $[0,t]$ into the union of the intervals  $\left[0, \frac{\lfloor nt\rfloor-1}{n}\right]$ and $\left[\frac{\lfloor nt\rfloor-1}{n},t\right]$. For the first interval we can write
\begin{align}
\int^{\frac{\lfloor nt\rfloor-1}{n}}_0\left[(t-\eta_n(s))^{\alpha}-(t-s)^{\alpha}\right]^2\, ds \nonumber
&\leq \sum_{i=0}^{\lfloor nt\rfloor-2}\int_{\frac{i}{n}}^{\frac{i+1}{n}}(t-s)^{2\alpha-2}(s-\eta_n(s))^2ds\\ \nonumber
&\leq\sum_{i=0}^{\lfloor nt\rfloor-2}\int_{\frac{i}{n}}^{\frac{i+1}{n}}(t-\frac{i+1}{n})^{2\alpha-2} n^{-2}ds \\  \nonumber
&\leq \sum_{i=0}^{\lfloor nt\rfloor-2}\frac{(nt-i-1)^{2\alpha-2}}{n^{2\alpha-2}} n^{-3}\\ \nonumber
&=\frac{1}{n^{2\alpha+1}}\sum_{i=0}^{\lfloor nt\rfloor-2}(nt-i-1)^{2\alpha-2}\\ \label{E1}
&  \leq \frac{1}{n^{2\alpha+1}} \sum_{j=1} ^\infty j^{2\alpha-2}.
\end{align}
Next for any $ s\in \left[\frac{\lfloor nt\rfloor-1}{n},t\right]$, we have
\begin{align*}
    \left[(t-\eta_n(s))^{\alpha}-(t-s)^{\alpha}\right]^2&\leq C\left[(t-\eta_n(s))^{2\alpha}+(t-s)^{2\alpha}\right] \\& \leq \begin{cases}
  C n^{-2\alpha} & \alpha \geq 0 \\
        C(t-s)^{2\alpha}& \alpha < 0 
    \end{cases} .
\end{align*}
Therefore,
 \begin{equation} \label{E2}
 \int_{\frac{\lfloor nt\rfloor-1}{n}}^t \left[(t-\eta_n(s))^{\alpha}-(t-s)^{\alpha}\right]^2ds \leq Cn^{-2\alpha-1}.
 \end{equation}
Thus combining \eqref{E1} and \eqref{E2} we obtain \eqref{ineq1}.  

The inequality \eqref{ineq2} follows easily from \eqref{e1}. Indeed, we can assume that $t>\delta$, and in this case
 \begin{align*}
    \int_{0}^{t-\delta}\left[(t-\eta_n(s))^{\alpha}-(t-s)^{\alpha}\right]^2\, ds&\leq \int_{0}^{t-\delta}(t-s)^{2\alpha-2}(s-\eta_n(s))^2ds\\
    &\leq \frac{1}{n^2}\int_{0}^{t-\delta}(t-s)^{2\alpha-2}ds\\
    &= \frac 1{n^2} \int_\delta ^t s^{2\alpha -2} ds \\
    &\leq \frac{2\delta^{2\alpha-1}}{n^2(1-2\alpha)} = C_2{n^{-2}} \delta^{2\alpha-1}.
\end{align*}
\end{proof}

 \begin{lemma} \label{lema1}
 Let $\psi_{n,2}(u,s)$ defined as in \eqref{psi2}. There exists a constant $C>0$ such that for any $s\ge 0$, we have
 \[
 n^{2\alpha+1} \int_0^{\eta_n(s)} \psi^2_{n,2}(u,s) du \le C. 
 \]
 \end{lemma}
 
 \begin{proof} We have, using the Mean Value Theorem,
 \begin{align*}
  \int_0^{\eta_n(s)} \psi^2_{n,2}(u,s) du &=  \int_0^{\eta_n(s)} [(s-\eta_n(u))^{\alpha}-(\eta_n(s)-\eta_n(u))^{\alpha}]^2 du\\
   & \le \alpha^2 (s- \eta_n(s))^2  \int_0^{\eta_n(s)} (\eta_n(s)-\eta_n(u))^{2\alpha-2} du \\
  & \le C n^{-3}  \sum_{i=0} ^{\lfloor ns \rfloor-1}    \left(\eta_n(s) -\frac in \right)^{2\alpha-2} \\
  & =  Cn^{-2\alpha-1}      \sum_{i=0} ^{\lfloor ns \rfloor-1} (\eta_n-i)^{2\alpha-2}.
  \end{align*}
 Since $2\alpha-2 <-1$ the series $\sum_{j=1} ^\infty j^{2\alpha-2}$ is convergent, and this completes the proof of the lemma.
 \end{proof}
 
 \begin{lemma} \label{lema2}
 Let $\psi_{n,1}(u,s)$ defined as in \eqref{psi1}. There exists a constant $C>0$ such that for any $s\ge 0$, we have
 \[
 n^{2\alpha+1} \int_0^{s} \psi^2_{n,1}(u,s) du \le C. 
 \]
 \end{lemma}
 
 \begin{proof} We have, using the Mean Value Theorem,
 \begin{align*}
  \int_0^{s} \psi^2_{n,1}(u,s) du &=  \int_0^{s} [(s-\eta_n(u))^{\alpha}-(s-u)^{\alpha}]^2 du\\
  &= \sum_{i=0} ^{\lfloor ns \rfloor}   \int_{\frac in} ^{\frac {i+1} n \wedge s}  [ (s-\frac in)^\alpha - (s-u)^\alpha]^2 du\\
  & =  n^{-2\alpha-1}\sum_{i=0} ^{\lfloor ns \rfloor}   \int_{i} ^{ (i+1) \wedge ns } [ (ns-i)^\alpha - (ns-x)^\alpha]^2 dx \\
    & \le  n^{-2\alpha-1}\sum_{k=0} ^{\infty}    \int_{0} ^{ 1 \wedge (ns -\lfloor ns \rfloor+k)} [ (ns-\lfloor ns \rfloor +k)^\alpha - (ns-\lfloor ns \rfloor +k-x)^\alpha]^2 dx,
  \end{align*}
which is uniformly bounded by a constant $C$ that does not depend on $n$ and $s$.
 \end{proof}

\end{document}